\documentclass[11pt]{amsart}
\usepackage[margin=1in]{geometry}
\usepackage{url}
\usepackage{hyperref}
\hypersetup{colorlinks=true, citecolor=blue}

\usepackage[lite]{amsrefs}
\usepackage{amssymb}
\usepackage{mathrsfs}
\usepackage[shortlabels]{enumitem}
\usepackage[all,cmtip]{xy}

\newcommand{\tensor}{\otimes}
\newcommand{\isom}{\cong}

\newcommand{\C}{\mathbb{C}}

\newcommand{\Q}{\mathbb{Q}}

\newcommand{\M}{\mathcal{M}}
\newcommand{\N}{\mathcal{N}}
\newcommand{\D}{\mathcal{D}}

\renewcommand{\H}{\mathcal{H}}
\newcommand{\I}{\mathcal{I}}
\renewcommand{\O}{\mathcal{O}}
\newcommand{\DR}{\mathrm{DR}}

\newcommand{\IC}{\mathcal{IC}}
\newcommand{\codim}{\mathrm{codim}}
\newcommand{\w}{\omega}
\newcommand{\W}{\mathcal{W}}

\newcommand{\gr}{\mathrm{gr}}
\newcommand{\Hom}{\mathcal{H}om}

\newcommand{\DB}{\underline{\Omega}} 
\newcommand{\Sing}{\mathrm{Sing}}
\newcommand{\dual}{\mathbf{D}}
\newcommand{\Supp}{\mathrm{Supp}}

\numberwithin{equation}{section}

\theoremstyle{plain}
\newtheorem{thm}[equation]{Theorem}
\newtheorem{cor}[equation]{Corollary}
\newtheorem{lem}[equation]{Lemma}
\newtheorem{prop}[equation]{Proposition}

\theoremstyle{definition}
\newtheorem{defn}[equation]{Definition}

\theoremstyle{remark}
\newtheorem{rmk}[equation]{Remark}
\newtheorem{ex}[equation]{Example}

\begin{document}

\title[Du Bois complex and extension of forms]{Du Bois complex and extension of forms beyond rational singularities}
\author{Sung Gi Park}
\address{Department of Mathematics, Harvard University, 1 Oxford Street, Cambridge, MA 02138, USA}
\email{sgpark@math.harvard.edu}


\date{\today}

\begin{abstract}
We establish a characterization of the Du Bois complex of a reduced pair $(X,Z)$ when $X\smallsetminus Z$ has rational singularities. As an application, when $X$ has normal Du Bois singularities and $Z$ is the locus of non-rational singularities of $X$, holomorphic $p$-forms on the smooth locus of $X$ extend regularly to forms on a resolution of singularities for $p\le\codim_X Z-1$, and to forms with log poles over $Z$ for $p\ge\codim_X Z$. If $X$ is not necessarily Du Bois, then $p$-forms extend regularly for $p\le\codim_X Z-2$. This is a generalization of the theorems of Flenner, Greb-Kebekus-Kov\'acs-Peternell, and Kebekus-Schnell on extending holomorphic (log) forms.

A by-product of our methods is a new proof of the theorem of Koll\'ar-Kov\'acs that log canonical singularities are Du Bois. We also show that the Proj of the log canonical ring of a log canonical pair is Du Bois if this ring is finitely generated. The proofs are based on Saito's theory of mixed Hodge modules.
\end{abstract}

\maketitle

\tableofcontents

\section{Introduction}
\label{sec:intro}

In this paper, we prove new results on the extension problem for holomorphic $p$-forms on a variety that recover and generalize some of the main existing results from the literature. The main idea lies in the delicate use of Saito's theory of mixed Hodge modules in order to enhance our understanding of Du Bois complexes. In this process, we also derive a streamlined proof of the theorem of Koll\'ar-Kov\'acs that log canonical implies Du Bois. We begin with the discussion of our main results on the extension of forms, following the conventions below.

Throughout the paper, a variety is a reduced, but not necessarily irreducible, separated scheme of finite type over $\C$. We call a variety $X$ \textit{embeddable} if $X$ admits a closed embedding into a smooth variety. For example, every variety is locally embeddable and every quasi-projective variety is embeddable. A log resolution of a pair $(X,Z)$ - consisting of a variety $X$ and a closed subvariety $Z$ - is a proper morphism $\mu:(\widetilde X,E)\to (X,Z)$ from a smooth variety $\widetilde X$ birational to $\overline {X\smallsetminus Z}$ (the closure of $X\smallsetminus Z$ in $X$) such that $E=\mu^{-1}(Z)$ (the set-theoretic preimage of $Z$) is a reduced simple normal crossing divisor.

\subsubsection*{\textnormal{\textbf{Extending holomorphic differential forms.}}}

The extension problem for holomorphic forms is a topic of significant importance in recent years. Specifically, for a normal variety $X$, we study the extendability of holomorphic $p$-forms on the smooth locus of $X$ to holomorphic $p$-forms on a desingularization $\widetilde X$, where $\mu:\widetilde X\to X$ is a resolution of singularities. In sheaf-theoretic language, this problem is asking for a condition under which the pushforward $\mu_*\Omega_{\widetilde X}^p$ of the sheaf of holomorphic $p$-forms is reflexive, in which case we say \textit{$p$-forms extend}. Analogously, given a subvariety $Z\subset X$, we say  \textit{$p$-forms extend with log poles over $Z$} if $\mu_*\Omega_{\widetilde X}^p(\log E)$ is reflexive for a log resolution $\mu:(\widetilde X,E)\to (X,Z)$. Notice that the pushforward sheaf $\mu_*\Omega_{\widetilde X}^p$ (or $\mu_*\Omega_{\widetilde X}^p(\log E)$) is independent of the choice of a (log) resolution $\mu$.

The following theorem unifies and generalizes the previous extension theorems of \cites{Flenner88, GK14, GK14b, GKKP11, KS21}.

\begin{thm}
\label{thm: extending forms}
Let $X$ be an irreducible normal variety and $Z\subset \mathrm{Sing}(X)$ be a closed subvariety such that $X\smallsetminus Z$ has rational singularities. Let $\mu:(\widetilde X,E)\to (X,Z)$ be a log resolution of singularities with $E=\mu^{-1}(Z)$. Suppose $X$ has Du Bois singularities away from a subvariety of codimension $m$. Then,
\begin{enumerate}
    \item the sheaf $\mu_*\Omega^p_{\widetilde X}$ is reflexive for all $0\le p\le\min\left\{\mathrm{codim}_XZ-1, m-2\right\}$, and
    \item the sheaf $\mu_*\Omega^p_{\widetilde X}(\log E)$ is reflexive for all $0\le p\le m-2$.\footnote{After the posting of this paper, I became aware of the preprint \cite{Tighe23} by Benjamin Tighe, proving this theorem in a special case when $Z=\mathrm{Sing}(X)$ and $X$ is Du Bois everywhere.}
\end{enumerate}
\end{thm}

In the above theorem, $\mathrm{Sing}(X)$ is the singular locus of $X$ and we follow the convention that the codimension (resp. dimension) of an empty set is infinity $\infty$ (resp. $-\infty$), and $\infty=\infty+1$. In particular, this applies when $Z$ is empty. In addition, notice that $m\ge \codim_XZ$ because rational singularities are Du Bois (see \cites{Kovacs99, Saito00}). 

Here is how this theorem recovers the existing results in the literature:

\begin{itemize}
\item If $Z=\mathrm{Sing}(X)$, then $p$-forms extend when $0\le p\le \codim_X \Sing(X)-2$ which is the theorem of Flenner \cite{Flenner88}.
\item If $X$ is Du Bois in codimension 2, then $1$-forms extend. This implies the result of Graf-Kov\'acs \cites{GK14, GK14b} which assumes that $X$ is Du Bois everywhere. As a consequence, the Lipman-Zariski conjecture (stating that $X$ is smooth if the tangent sheaf is locally free) holds if $X$ is Du Bois in codimension $2$ (see e.g. \cite{GKKP11}*{Theorem 6.1} for the proof that extending of $1$-forms implies the conjecture).
\item If $X$ has rational singularities and $Z=\emptyset$, then $p$-forms extend for all $0\le p\le\dim X$ which is the theorem of Kebekus-Schnell \cite{KS21}.
\item If $X$ has log canonical singularities, then $p$-forms extend with log poles over the non-klt locus $nklt(X)$ for all $0\le p\le\dim X$, which is the theorem of Greb-Kebekus-Kov\'acs-Peternell \cite{GKKP11}*{Theorem 1.5} when the boundary divisor $D=0$. Indeed, we take $Z=nklt(X)$ and apply (2). For details, see the discussion below and Corollary \ref{cor: extension for log canonical}. The case when $D\neq 0$ is addressed in Corollary \ref{cor: GKKP and vanishing cohomology}.
\end{itemize}

We emphasize one of the main improvements in Theorem \ref{thm: extending forms}: $Z$ can be chosen as the complement of the locus of rational singularities, which is generally smaller than $\Sing(X)$ or $nklt(X,D)$ (i.e. the non-klt locus of a pair $(X,D)$ of a normal variety $X$ and a $\Q$-divisor $D$).

Besides the extension results mentioned above, this theorem covers significant new cases. As the simplest example, if a pair $(X,D)$ is log canonical in codimension $2$, then $1$-forms extend; we take $Z=\Sing (X)$. Here, we are implicitly using the theorem of Koll\'ar-Kov\'acs \cite{KK10} that log canonical singularities are Du Bois, which is discussed more thoroughly in the next subsection. Notably, for log canonical pairs, we have the following result, which readily follows from a well-known generalization of the theorem of Elkik \cite{Elkik81} stating that Kawamata log terminal singularities are rational, combined with Theorem \ref{thm: extending forms}.

\begin{cor}
\label{cor: extension for log canonical}
Let $(X,D)$ be a log canonical pair with an effective $\Q$-divisor $D$. Let $Z\subset nklt(X,D)\cap \Sing(X)$ be the complement of the locus of rational singularities. Then $p$-forms extend for $p \le \codim_XZ-1$, and $p$-forms extend with log poles over $Z$ for $p\ge \codim_XZ$.

In particular, if $X$ is log canonical everywhere and Kawamata log terminal away from a point, then $p$-forms extend for all $p\le \dim X-1$.
\end{cor}

The proof of Theorem \ref{thm: extending forms} is a consequence of the following two technical results on extending (log) forms, combined with the characterization of Du Bois complexes that we establish in the next subsection.

\begin{thm}
\label{thm: p-forms extend}
In the setting of Theorem \ref{thm: extending forms}, the sheaf $\mu_*\Omega_{\widetilde X}^p$ is reflexive for all $0\le p\le\codim_XZ-2$. Furthermore, if $R^{\codim_X Z-1}\mu_*\O_{\widetilde X}(-E)$ is generically vanishing on the components of $Z$ of largest dimension, then $\mu_*\Omega_{\widetilde X}^p$ is reflexive for all $0\le p\le\codim_XZ-1$.
\end{thm}

\begin{rmk}
\label{rmk: weakly rational case}
This theorem is true, with exactly the same proof, under the weaker assumption that $X\smallsetminus Z$ has weakly rational singularities (see \cite{KS21}*{Definition A.1}), or equivalently
$$
\codim_X\left(\Supp \left(R^{i-2}\mu_*\O_{\widetilde X\smallsetminus E}\right)\right)\ge i \mathrm{\;\;for\;all\;\;} i\ge 3.
$$
This characterization of weakly rational singularities follows from \eqref{eqn: duality for IC}, Lemma \ref{lem: basic IC}, and \cite{HL10}*{Proposition 1.1.6 (ii)}.
\end{rmk}

The second technical result for the proof of Theorem \ref{thm: extending forms} regards the extension of forms with log poles, without singularity conditions on the complement:

\begin{thm}
\label{thm: p-log form extends}
    Let $X$ be an irreducible normal variety of dimension $n$ and $Z$ be a subvariety. Let $\mu:(\widetilde X,E)\to (X,Z)$ be a log resolution with $E=\mu^{-1}(Z)$ and
    $$
    m=\max\left\{k\;|\;\codim_X\left(\Supp \left(R^{i-2}\mu_*\O_{\widetilde X}(-E)\right)\right)\ge i \mathrm{\;\;for\;all\;\;} 3\le i\le k\right\}.
    $$
    Then, $\mu_*\Omega^p_{\widetilde X}(\log E)$ is reflexive for all $0\le p\le m-2$.
\end{thm}

The proofs of Theorems \ref{thm: p-forms extend} and \ref{thm: p-log form extends} are strongly inspired by the arguments of Kebekus-Schnell \cite{KS21}. The principal addition to these arguments is the use of a mixed Hodge module $\N_{X,Z}:=\H^0(j_!\IC^H_{X\smallsetminus Z})$ and its properties deduced from the proof of Theorem \ref{thm: Du Bois complex, rational singularities} below, which allows one to establish a connection between $R\mu_*\O_{\widetilde X}(-E)$ and $\mu_*\Omega^p_{\widetilde X}(\log E)$. See Section \ref{sec: extending forms} for details.

Lastly, we recover the theorem of Greb-Kebekus-Kov\'acs-Peternell \cite{GKKP11}*{Theorem 1.5}, stating that for a log canonical pair $(X,D)$, $p$-forms extend with log poles over $Z=nklt(X,D)$ for all $0\le p\le \dim X$. It is worth noting that both $X$ and $Z$ are Du Bois \cite{KK10} and $X\setminus Z$ is rational. Therefore, we have the following quasi-isomorphisms
$$
\I_{X,Z}\isom\DB^0_{X,Z}\isom R\mu_*\O_{\widetilde X}(-E)
$$
that follow from Theorem \ref{thm: Du Bois complex, rational singularities} below (see the next subsection for the notation). Combined with Theorem \ref{thm: p-log form extends}, we have

\begin{cor}
\label{cor: GKKP and vanishing cohomology}
Let $(X,D)$ be a log canonical pair with an effective $\Q$-divisor $D$ and
$$
\mu:(\widetilde X,E)\to (X,nklt(X,D))
$$
be a log resolution of singularities with $E=\mu^{-1}(nklt(X,D))$. Then $R^i\mu_*\O_{\widetilde X}(-E)=0$ for all $i>0$ and $\mu_*\Omega^p_{\widetilde X}(\log E)$ is reflexive for all $0\le p\le \dim X$.
\end{cor}

\begin{rmk}
Notice that $(X,Z):=(X,nklt(X,D))$ is an example of a log rational pair defined below (see Definition \ref{defn: log rational pair}). It is easy to see that the conclusion of Corollary \ref{cor: GKKP and vanishing cohomology} holds for any log rational pair $(X,Z)$ with $X$ normal and irreducible.
\end{rmk}

\subsubsection*{\textnormal{\textbf{New characterization of Du Bois complexes.}}} 

One of the crucial tools for proving the preceding extension theorems is a characterization of Du Bois complexes of certain pairs, discussed next. The Du Bois complex $\DB^0_X$ of a variety $X$ has gained significant attention due to its role in higher dimensional geometry, particularly in connection with the compactification of the moduli space of stable varieties. Despite its intricate definition, substantial progress has been made in understanding it over the last few decades (see, for example, \cites{Kovacs99, Kovacs00b, Schwede07, KK10, Saito00}). In particular, Saito \cite{Saito00} proved that the Du Bois complex is the graded de Rham complex of a complex of mixed Hodge modules (see Section \ref{sec: Saito MHM}). 

For our purposes, we are interested in the following variant for a pair $(X,Z)$ consisting of a variety $X$ and a closed subvariety $Z$. As defined by Steenbrink \cite{Steenbrink85}, its Du Bois complex $\DB^0_{X,Z}$ is the mapping cone of the natural morphism $\rho:\DB^0_X\to\DB^0_Z$ of Du Bois complexes shifted by $[-1]$. Alternatively, $\DB^0_{X,Z}$ is an object in the bounded derived category of coherent sheaves $D^b_{coh}(X, \O_X)$, sitting in a distinguished triangle:
$$
\DB^0_{X,Z}\rightarrow\DB^0_X\xrightarrow{\rho}\DB^0_Z\xrightarrow{+1}.
$$
Recall that $X$ has Du Bois singularities if the natural morphism $\O_X\to \DB^0_X$ is a quasi-isomorphism and $(X,Z)$ is a \textit{Du Bois pair} if the natural morphism $\I_{X,Z}\to \DB^0_{X,Z}$ is a quasi-isomorphism where $\I_{X,Z}$ is the ideal sheaf of $Z$ in $X$ (see \cite{Kovacs11}*{Definition 3.13}). In particular, $(X,Z)$ is a Du Bois pair if $X$ and $Z$ have Du Bois singularities.

We will study the Du Bois complex of a pair using the theory of mixed Hodge modules. When dealing with complexes of mixed (versus pure) Hodge modules, the powerful tools that have played a crucial role in the recent success of integrating Saito's Hodge module theory in higher dimensional geometry are no longer readily available. Notably, the Decomposition Theorem, a significant property of pure Hodge modules, is inaccessible in this context. To overcome this limitation, we study some key properties of the graded de Rham complex of mixed Hodge modules in Section \ref{sec: graded de Rham} and use them in Sections \ref{sec: rational singularities}, \ref{sec: characterization of Du Bois complexes} to provide a particularly simple characterization of the Du Bois complex of a pair when the complement has rational singularities.

\begin{thm}
\label{thm: Du Bois complex, rational singularities}
    Let $X$ be an embeddable variety and $Z\subset X$ be a closed subvariety such that $X\smallsetminus Z$ has rational singularities. Let $\mu:(\widetilde X,E)\to (X,Z)$ be a log resolution of singularities with $E=\mu^{-1}(Z)$. Then there exist quasi-isomorphisms
    $$
    \DB^0_{X,Z}\isom R\mu_*\O_{\widetilde X}(-E)\isom \gr^F_0\DR\left(j_!\left(\bigoplus_i\IC_{X_i\smallsetminus Z}^H[-\dim X_i]\right)\right)
    $$
    where $j:X\smallsetminus Z\to X$ is the open embedding and $\IC_{X_i}^H$ is the pure polarizable Hodge module associated to the intersection complex of $X_i$ for each irreducible component $X_i$ of $\overline {X\smallsetminus Z}$.
\end{thm}

For the third term of the quasi-isomorphisms, see Section \ref{sec: Saito MHM} which contains a discussion of the fundamentals of Saito's theory of mixed Hodge modules and the graded de Rham functor $\gr^F_p\DR$.

The main difficulty in the proof of Theorem \ref{thm: Du Bois complex, rational singularities} comes from the fact that $\mu$ is not an isomorphism over $X\smallsetminus Z$, setting this theorem apart from the characterizations of Du Bois complexes found in the literature (see e.g. \cite{Schwede07}). It is exactly for this reason that the use of mixed Hodge modules becomes necessary.  Note that we do not claim the naturality of the quasi-isomorphism $\DB^0_{X,Z}\isom R\mu_*\O_{\widetilde X}(-E)$ (see Remark \ref{rmk: naturality}).

\subsubsection*{\textnormal{\textbf{Log rational pairs and the proof of log canonical implies Du Bois.}}} 

The techniques used to prove Theorem \ref{thm: Du Bois complex, rational singularities} also lead to novel proofs of the inversion of adjunction for rational and Du Bois singularities, as well as the celebrated theorem of Koll\'ar-Kov\'acs, stating that log canonical singularities are Du Bois. The key concept underlying these proofs is the following:

\begin{defn}
\label{defn: log rational pair}
    A pair $(X,Z)$ is called a \textit{log rational pair} if it is a Du Bois pair and $X\smallsetminus Z$ has rational singularities.
\end{defn}

Using Theorem \ref{thm: Du Bois complex, rational singularities}, log rational pairs have a nice characterization in terms of resolution of singularities. Note that in the result below, the converse statement is an immediate consequence of the results of Kov\'acs \cites{Kovacs00, Kovacs11}.

\begin{cor}
\label{cor: criterion for log rational pair}
    If a pair $(X,Z)$ is a log rational pair, then the natural morphism
    $$
    \I_{X,Z}\to R\mu_*\O_{\widetilde X}(-E)
    $$
    is a quasi-isomorphism, where $\mu:(\widetilde X,E)\to (X,Z)$ is a log resolution of singularities with $E=\mu^{-1}(Z)$. Conversely, for any proper morphism $f:(X',Z')\to (X,Z)$ such that $(X',Z')$ is a log rational pair with $f(Z')\subset Z$, if the natural morphism $\I_{X,Z}\to Rf_*\I_{X',Z'}$ admits a left quasi-inverse (that is a morphism $Rf_*\I_{X',Z'}\to \I_{X,Z}$ such that their composition is a quasi-isomorphism of $\I_{X,Z}$ to itself), then $(X,Z)$ is a log rational pair.
\end{cor}

\begin{rmk}
The notion of a rational pair was defined by Koll\'ar, see \cite{Kollar13}*{Definition 2.80}, and using Kov\'acs \cite{Kovacs11}*{Theorem 1.1} we can easily see that a rational pair is log rational. However, the notion of a log rational pair is more general.
\end{rmk}

One immediate consequence of Corollary \ref{cor: criterion for log rational pair} is the inversion of adjunction for rational and Du Bois singularities, proposed as a conjecture by Kov\'acs-Schwede \cite{KS11}*{Conjecture 12.5}. Ma-Schwede-Shimomoto \cite{MSS17}*{Corollary 3.10} proved a stronger version involving rational pairs (see \cite{KS16}*{Conjecture 7.9}). We give an alternative concise proof using Corollary \ref{cor: criterion for log rational pair}.

\begin{cor}[Inversion of adjunction]
\label{cor: inversion of adjunction}
    Let $X$ be a variety and $D$ be a reduced Cartier divisor. If $X\smallsetminus D$ has rational singularities and $D$ has Du Bois singularities, then $X$ has rational singularities.
\end{cor}

As a by-product of Corollary \ref{cor: criterion for log rational pair}, we also give a new Hodge theoretic proof of the following generalization of the theorem of Koll\'ar-Kov\'acs \cite{KK10} stating that log canonical singularities are Du Bois. Notationwise, a sub-boundary $\Q$-divisor $\Delta$ is a divisor whose coefficients are rational numbers at most $1$. If $\Delta$ is additionally effective, then we call it a boundary $\Q$-divisor.

\begin{thm}
\label{thm: lc implies DB}
Let $(X,\Delta)$ be a log canonical pair with a sub-boundary $\Q$-divisor $\Delta$. Suppose there is a proper morphism $f:X\to Y$ such that $K_X+\Delta\sim_{\Q,f}0$ and $f_*\O_X(\lceil-\Delta^{<1}\rceil)= \O_Y$. Then $Y$ has Du Bois singularities.

Furthermore, if $W\subset Y$ is a union of $f$-images of log canonical centers of $(X,\Delta)$, then $W$ has Du Bois singularities, and if $W$ is minimal, then $W$ has rational singularities.
\end{thm}

Here, $W$ is minimal if there is no $f$-image of a log canonical center of $(X,\Delta)$ properly contained in $W$. In addition, $\Delta^{<1}$ denotes the partial sum of components of $\Delta$ with coefficients less than $1$.

This theorem is also implied by the results of Fujino-Liu \cite{FL22}*{Theorem 1.1} that quasi-log canonical pairs are Du Bois, and by Fujino \cite{Fujino22}*{Corollary 1.10} that minimal quasi-log canonical strata have rational singularities (see also \cite{Fujino99}*{Theorem 1.2} when $(X,\Delta)$ is klt). While the previous proofs in \cite{KK10} and \cites{Fujino22,FL22} are based on either the minimal model program or the theory of quasi-log canonical pairs, the proof given here is based on the theory of mixed Hodge modules. A sketch for a special situation ($X$ is Gorenstein, $\Delta=0$, and $Y=X$) is provided at the beginning of Section \ref{sec: log canonical and Du Bois}.

Furthermore, we prove a result regarding the singularity of the Proj of a log canonical ring.

\begin{cor}
\label{cor: log canonical model is Du Bois}
Let $(X,D)$ be a proper log canonical pair (resp. klt pair) with a boundary $\Q$-divisor $D$. If the log canonical ring $R(X,D)$ is finitely generated, then $\mathrm{Proj}\; R(X,D)$ has Du Bois singularities (resp. rational singularities).

Furthermore, there exists an increasing sequence of subvarieties
$$
\emptyset=Y_{k+1}\subsetneq Y_{k}\subsetneq \dots\subsetneq Y_1\subsetneq Y_0=\mathrm{Proj}\; R(X,D)
$$
such that $(Y_i,Y_{i+1})$ is a log rational pair for all $i\in [0,k]$.
\end{cor}

For klt pairs, this result is established in Nakayama \cite{Nakayama88}*{Theorem} (when $D=0$) and Fujino-Mori \cite{FM00}*{Theorem 5.4}, with different proofs. For log canonical pairs, Fujino proved in the proof of \cite{Fujino15}*{Theorem 1.6} that $\mathrm{Proj}\; R(X,D)$ is quasi-log canonical, and hence Du Bois by \cite{FL22}, although this is not explicitly stated.

\subsubsection*{\textnormal{\textbf{Overview of the paper.}}}

The background needed for Saito's theory is laid out in Sections \ref{sec: Saito MHM} and \ref{sec: graded de Rham}. This is followed by the analysis of rational singularities in Section \ref{sec: rational singularities}. Section \ref{sec: characterization of Du Bois complexes} establishes the characterization of the Du Bois complex of a pair when the complement has rational singularities, namely Theorem \ref{thm: Du Bois complex, rational singularities}.

The remainder of the paper focuses on the main results and applications. Section \ref{sec: extending forms} deals with the extension theorems on (log) forms. After developing the necessary technicalities, we show in Section \ref{sec: log canonical and Du Bois} that log canonical singularities are Du Bois and study related problems.

\subsubsection*{\textnormal{\textbf{Acknowledgements.}}}

I am grateful to my advisor, Mihnea Popa, for his support and valuable conversations. I would like to thank Osamu Fujino, Joe Harris, S\'andor Kov\'acs, Mircea Mustață, Christian Schnell, Wanchun Shen, and Ahn Duc Vo for their insightful comments and helpful discussions. I am supported by the 17th Kwanjeong Study Abroad Scholarship.

\section{Overview of Saito's theory of mixed Hodge modules}
\label{sec: Saito MHM}

Saito \cites{Saito88, Saito90} introduced the theory of mixed Hodge modules on any singular variety $X$. Given an embedding $X\subset \W$ into a smooth variety, a mixed Hodge module $\M\in MHM(X)$ has the underlying filtered regular holonomic $\D$-module with $\Q$-structure,
\begin{equation}
\label{eqn: MHM structure}
   (M,F_\bullet M, K;W), 
\end{equation}
where $M$ is a regular holonomic (right) $\D_{\mathcal W}$-module, $F$ is an increasing Hodge filtration on $M$, $K\in Perv(X;\Q)$ is a perverse sheaf on $X$ with coefficients in $\Q$ with an isomorphism $\DR_\W(M)\isom K\tensor_\Q \C$, and $W$ is a weight filtration on the structure $(M,F_\bullet M,K)$. For convenience, throughout the text, we also denote the underlying filtered regular holonomic $\D$-module by $\M$.

Recall that for a regular holonomic (right) $\D_\W$-module $M$, there exists a de Rham functor
$$
\DR_\W(M):=M^{an}\tensor^L_{\D_\W^{an}}\O_\W^{an}\in Perv(\W;\C),
$$
where $(\bullet)^{an}$ is the analytification functor (see for example \cite{HTT}*{Section 4.7}). This can be naturally represented as the analytification of a complex
\begin{equation}
\label{eqn: de Rham complex}
0\to M\tensor_{\O_\W}\wedge^{\dim \W}T_\W\to\cdots\to M\tensor_{\O_\W}T_\W\to M\to 0,
\end{equation}
induced by the Spencer resolution of the left $\D_\W$-module $\O_\W$ (c.f. \cite{HTT}*{Lemma 1.5.27}). Here, $T_\W$ is the tangent sheaf of $\W$ and the rightmost $M$ is supported at degree $0$. For example, when $M$ is the right $\D_\W$-module $\w_\W:=\det\Omega_\W$, we obtain the usual de Rham complex of $\W$, which is quasi-isomorphic to the perverse sheaf $\C_\W[\dim \W]$.

For a mixed Hodge module $\M\in MHM(X)$, its de Rham complex $\DR_\W(\M)$ is defined by $\DR_\W(M)$, which has an induced Hodge filtration $F$ on the complex \eqref{eqn: de Rham complex} defined by
\begin{equation}
\label{eqn: filtration on DR}
F_p\left(M\tensor_{\O_\W}\wedge^{-i}T_\W\right)=F_{p+i}M\tensor_{\O_\W}\wedge^{-i}T_\W,
\end{equation}
for each $i$-th degree component of \eqref{eqn: de Rham complex}. It is well-known that the graded complex $\gr^F_p\DR_\W(\M)$ is the complex of coherent $\O_X$-modules. This is well-defined for the complex of mixed Hodge modules $\M^\bullet$, thereby inducing the graded de Rham functor
$$
\gr^F_p\DR_\W:D^bMHM(X)\to D^b_{coh}(X,\O_X)
$$
where $D^bMHM(X)$ is the bounded derived category of mixed Hodge modules on $X$ and $D^b_{coh}(X,\O_X)$ is the bounded derived category of $\O_X$-modules with coherent cohomologies. This functor is independent of the choice of an embedding $X\subset \W$ (see \cite{Saito90}*{Proposition 2.33}). Accordingly, we omit the subscript $\W$ and denote this functor $\gr^F_p\DR$.

One of the main accomplishments of Saito's theory is the construction of natural functors compatible with those for underlying $\Q$-complexes.

\begin{thm}[\cite{Saito90}*{Theorem 0.1}]
\label{thm: Saito main}
Let $X$ be an algebraic variety. Then we have a dualizing functor
\begin{gather*}
\dual:MHM(X)\to MHM(X)^{\mathrm{op}}.
\end{gather*}
Let $f:X\to Y$ be a morphism of algebraic varieties. Then we have natural functors
$$
f_*,f_!:D^bMHM(X)\to D^bMHM(Y),\quad f^*, f^!:D^bMHM(Y)\to D^bMHM(X).
$$
These functors are compatible with the corresponding functors for the underlying $\Q$-complexes.
\end{thm}

The functors in Theorem \ref{thm: Saito main} exhibit various properties analogous to those of corresponding functors in the bounded derived category $D^b_c(X,\Q)$ of $\Q$-sheaves on $X$ with constructible cohomologies. For instance, if $g:Y\to Z$ is a morphism of algebraic varieties, then we have a canonical isomorphism
$$
(g\circ f)_*=g_*\circ f_*: D^bMHM(X)\to D^bMHM(Z)
$$
and likewise $(g\circ f)_!=g_!\circ f_!$, $(g\circ f)^*=f^*\circ g^*$, and $(g\circ f)^!=f^!\circ g^!$. Moreover, we emphasize the following fact that is used throughout the text.

\begin{prop}[\cite{Saito90}*{Section 4}]
We have the following canonical isomorphisms of the composition of functors:
$$
\dual\circ\dual=\mathrm{Id}, \quad \dual\circ f_*=f_!\circ\dual,\quad \dual\circ f^*=f^!\circ\dual.
$$
If $f$ is proper, then $f_*=f_!$. Furthermore, if there exists a cartesian square,
\begin{displaymath}
\xymatrix{
{X'}\ar[r]^{g'}\ar[d]_{f'}& {X}\ar[d]^{f}\\
{Y'}\ar[r]^{g}&{Y}
}
\end{displaymath}
then we have $g^!\circ f_*=f'_*\circ (g')^!$ and $g^*\circ f_!=f'_!\circ (g')^*$, compatible with the isomorphisms for the underlying $\Q$-complexes.
\end{prop}

Additionally, we note the following distinguished triangles, used later in the paper.

\begin{prop}[\cite{Saito90}*{4.4.1}]
\label{prop: Saito triangle}
Let $i:Z\to X$ be a closed embedding and $j:X\smallsetminus Z\to X$ be an open embedding. Then for any $\M^\bullet\in D^bMHM(X)$, we have distinguished triangles
\begin{gather*}
\to j_!j^!\M^\bullet\to \M^\bullet\to i_*i^*\M^\bullet\xrightarrow{+1}\\
\to i_!i^!\M^\bullet\to \M^\bullet\to j_*j^*\M^\bullet\xrightarrow{+1}
\end{gather*}
compatible with the corresponding distinguished triangles for the underlying $\Q$-complexes.
\end{prop}

For proofs of the above results and additional properties of $D^bMHM(X)$, see \cite{Saito90}*{Section 4}.

\subsubsection*{\textnormal{\textbf{Du Bois complexes via Saito's mixed Hodge modules}}}

The category of mixed Hodge modules $MHM(pt)$ of a point $pt:=\mathrm{Spec}\,\C$ is equivalent to the category of $\Q$-mixed Hodge structures (see \cite{Saito90}*{Theorem 3.9}). In particular, a point $pt$ has a mixed Hodge module 
$$
\Q_{pt}^H:=(\C, F_\bullet\C,\Q;W)
$$
such that $F_0\C=\C$, $F_{-1}\C=0$, and $W_0\Q=\Q$, $W_{-1}\Q=0$ (i.e. pure of weight $0$). Moreover, we define
$$
\Q^H_X:=a_X^*\Q_{pt}^H\in D^bMHM(X)
$$
where $a_X:X\to pt$. Theorem \ref{thm: Saito main} implies that the underlying $\Q$-complex of $\Q^H_X$ is the constant sheaf $\Q_X$. Consequently, the cohomologies of the direct image $(a_X)_*\Q^H_X\in D^bMHM(pt)$ induce mixed Hodge structures on the ordinary cohomologies of $X$ with rational coefficients. These mixed Hodge structures coincide with Deligne's mixed Hodge structures \cite{Deligne74} by the result of Saito \cite{Saito00}. In the process of proving this result, Saito proved that the Deligne-Du Bois complexes coincide with the graded de Rham complexes of objects in $D^bMHM(X)$. Specifically, we have

\begin{thm}[\cite{Saito00}*{Corollary 0.3, Theorem 4.2}]
\label{thm: Du Bois via MHM}
    Let $Z$ be a closed subvariety of $X$ and $j:X\smallsetminus Z\to X$ be an open embedding. Then, there exists a natural isomorphism
    $$
    \gr^F_0\DR\left(j_!\Q^H_{X\smallsetminus Z}\right)=\DB_{X,Z}^0.
    $$
\end{thm}

\begin{rmk}
In fact, Saito constructed the triangulated category of mixed $\Q$-Hodge $\mathcal D$-complexes $D^b_{\mathcal H}(X,\Q)_{\mathcal D}$ and a natural functor 
$$
\epsilon: D^bMHM(X)\to D^b_{\mathcal H}(X,\Q)_{\mathcal D}.
$$
Deligne's mixed $\Q$-Hodge complex, namely $\mathcal C^H_{X,Z}\Q$, with the underlying $\Q$-complex $j_!\Q_{X\smallsetminus Z}$ has an associated object $\DR_X^{-1}\left(\mathcal C^H_{X,Z}\Q\right)\in D^b_{\mathcal H}(X,\Q)_{\mathcal D}$  with a $\D$-module structure, and Saito proves a natural isomorphism
$$
\epsilon\left(j_!\Q^H_{X\smallsetminus Z}\right)=\DR_X^{-1}\left(\mathcal C^H_{X,Z}\Q\right)
$$
in $D^b_{\mathcal H}(X,\Q)_{\mathcal D}$. Theorem \ref{thm: Du Bois via MHM} is an immediate consequence obtained from taking the graded de Rham functor $\gr^F_0\DR$ on both sides. See \cite{Saito00} for details.
\end{rmk}

\subsubsection*{\textnormal{\textbf{Polarizable Hodge modules.}}} The category of mixed Hodge modules $MHM(X)$ admits a subcategory of polarizable Hodge modules of weight $w$, namely $MH(X,w)$. This consists of an object $\M\in MHM(X)$ such that $\gr^W_k\M=0$ for all $k\neq w$. Additionally, we say an object $\M^\bullet\in D^bMHM(X)$ is \textit{pure of weight $w$} if its cohomologies satisfy $\H^q(\M^\bullet)\in MH(X,w+q)$ for all $q$. In such case, we have the following decomposition.

\begin{prop}[\cite{Saito90}*{4.5.4}]
\label{prop: decomposition for pure weight}
If $\M^\bullet\in D^bMHM(X)$ is pure of weight $w$, then there exists a (non-canonical) isomorphism
$$
\M^\bullet\isom \bigoplus_q\H^q(\M^\bullet)[-q].
$$
\end{prop}

To explain the notion of polarization of weight $w$, we first recall the Tate twist of mixed Hodge modules. Continuing with \eqref{eqn: MHM structure}, given an integer $k$ and $\M\in MHM(X)$, the Tate twist $\M(k)\in MHM(X)$ consists of the following structure
$$
(M,F_{\bullet-k}M,K(k); W_{\bullet+2k}),
$$
where $K(k)=K\tensor_\Q\Q(k)$ and $\Q(k)=(2\pi i)^k\Q\subset \C$, the usual Tate twist for $\Q$-complexes.

A polarizable Hodge module $\M\in MH(X,w)$ of weight $w$ admits a polarization \cite{Saito88}*{5.2.10}, which consists of an isomorphism $\M(w)\isom \dual \M$ or equivalently
$$
(M,F_{\bullet-w} M, K(w))\isom \dual(M,F_\bullet M, K).
$$
See \cite{Saito88}*{2.4} for the definition of the duality functor $\dual$ on the derived category of filtered $\D$-modules, which induces the duality functor on $D^bMHM(X)$.

For example, let $\W$ be a smooth variety and $(\mathbb V, F^\bullet)$ be an irreducible polarizable variation of $\Q$-Hodge structures of weight $w-\dim \W$ on $\W$. Here, $\mathbb V$ is a $\Q$-local system and $F^\bullet$ is a decreasing Hodge filtration. Let $V$ be the corresponding (right) $\D_\W$-module of $\mathbb V$ with the increasing Hodge filtration $F_\bullet V$. Then
$$
(V, F_\bullet V, \mathbb V[\dim \W])
$$
is an irreducible polarizable Hodge module of weight $w$. See \cite{Saito88}*{Th\'eor\`eme 5.4.3} for the proof.

In general, the category $MH(X,w)$ of polarizable Hodge modules of weight $w$ on $X$ is semisimple. The simple objects of $MH(X,w)$ are \textit{minimal extensions} $\IC^H_Z(V_i)$ (see e.g. \cite{Saito88}*{Lemme 5.1.10}) of irreducible polarizable variations of $\Q$-Hodge structures $\mathbb V_i$ of weight $w-\dim Z$ defined over some smooth Zariski open subset of an irreducible subvariety $Z\subset X$. In other words, $\M\in MH(X,w)$ admits the decomposition
$$
\M=\bigoplus_{Z\subset X}\bigoplus_{i}\IC^H_Z(\mathbb V_i)
$$
where $Z\subset X$ ranges over the set of irreducible subvarieties of $X$. We say $Z$ is the \textit{strict support} of $\IC^H_Z(\mathbb V_i)$. In particular, every irreducible variety $X$ has the polarizable Hodge module $\IC^H_X$ of weight $\dim X$ associated to the intersection complex $\IC_X$ of $X$.

We end this section with Saito's theorem that generalizes Beilinson-Bernstein-Deligne-Gabber's Decomposition Theorem \cite{BBD}. This is one of the most important and useful properties of pure Hodge modules.

\begin{thm}[Saito's Decomposition Theorem]
\label{thm: Saito's decomposition theorem}
Let $f:X\to Y$ be a proper morphism of varieties. If $\M^\bullet\in D^bMHM(X)$ is pure of weight $w$, then $f_*\M^\bullet\in D^bMHM(Y)$ is pure of weight $w$, hence admits a non-canonical isomorphism
$$
f_*\M^\bullet\isom \bigoplus_q\H^q(f_*\M^\bullet)[-q].
$$
\end{thm}

\section{Graded de Rham complexes of mixed Hodge modules}
\label{sec: graded de Rham}

The proofs of the main results rely on an in-depth understanding of the graded de Rham functor $\gr^F_p\DR$. In preparation, we study its various properties starting with the commutativity of the direct image functor $f_*$ and the graded de Rham functor  under a proper morphism $f$:

\begin{prop}
\label{prop: commutativity of graded de Rham}
If $f:X\to Y$ is a proper morphism, then we have a natural isomorphism of functors
$$
\gr^F_p\DR\circ f_*=Rf_*\circ\gr^F_p\DR,
$$
or equivalently, the following diagram commutes.
\begin{displaymath}
\xymatrix{
{D^bMHM(X)}\ar[r]^{\gr^F_p\DR}\ar[d]_{f_*}& {D^b_{coh}(X,\O_X)}\ar[d]^{Rf_*}\\
{D^bMHM(Y)}\ar[r]^{\gr^F_p\DR}&{D^b_{coh}(Y,\O_Y)}
}
\end{displaymath}
\end{prop}

\begin{proof}
This is an immediate consequence of \cite{Saito00}*{Proposition 2.8}. Alternatively, this follows from the construction of $f_*$: by Beilinson's argument in \cite{Beilinson87}, an object $\M^\bullet\in D^bMHM(X)$ is isomorphic to the complex of mixed Hodge modules such that each component is $f_*$-acyclic. Assume this is the case. Then, $f_*\M^\bullet$ is isomorphic to the complex of mixed Hodge modules whose $q$-th component is $\H^0(f_*\M^q)$. Consequently from \cite{Saito88}*{2.3.7}, we have a natural isomorphism
$$
\gr^F_p\DR( f_*\M^\bullet)=Rf_*\left(\gr^F_p\DR(\M^\bullet)\right).
$$
This completes the proof.
\end{proof}

The dualizing functor $\dual$ combined with the graded de Rham functor on $D^bMHM(X)$ induces the coherent duality in $D^b_{coh}(X,\O_X)$:

\begin{lem}
\label{lem: duality}
Let $X$ be an embeddable variety and $\M^\bullet\in D^bMHM(X)$. Then, for every integer $p$, we have an isomorphism
$$
R\Hom_{\O_X}\left(\gr^F_p\DR(\M^\bullet),\omega_X^\bullet\right)\isom\gr^F_{-p}\DR(\mathbb \dual(\M^\bullet))
$$
in $D^b_{coh}(X,\O_X)$, where $\omega^\bullet_X$ is the dualizing complex of $X$.
\end{lem}

\begin{proof}
Although not stated explicitly, this is proved in \cite{Saito88}*{Section 2.4} and used in the proof of \cite{Saito00}*{Proposition 5.3}. Suppose we have an embedding $X\subset \W$ into a smooth variety. From the definition of the dual functor, we have a natural isomorphism
$$ R\Hom_{\O_\W}\left(\gr^F_p\DR(\M^\bullet),\omega_\W^\bullet\right)\isom\gr^F_{-p}\DR(\mathbb \dual(\M^\bullet))
$$
in $D^b_{coh}(X,\O_X)$, and the left hand side is isomorphic to $R\Hom_{\O_X}\left(\gr^F_p\DR(\M^\bullet),\omega_X^\bullet\right)$ by Grothendieck duality.
\end{proof}


In particular, for a polarizable Hodge module $\IC_X^H\in MH(X,n)$ of weight $n=\dim X$ associated to the intersection complex of $X$, we have the polarization $\IC_X^H(n)\isom \dual \IC_X^H$. This induces the duality
\begin{equation}
\label{eqn: duality for IC}
R\Hom_{\O_X}\left(\gr^F_p\DR\left(\IC^H_X\right),\omega_X^\bullet\right)\isom\gr^F_{-p-n}\DR\left(\IC^H_X\right).
\end{equation}
See also \cite{KS21}*{Proposition 4.4}.

For a nonzero object $\M^\bullet\in D^bMHM(X)$, we define the \textit{index of the first nonzero term in the Hodge filtration} of $\M^\bullet$ as the minimal $p$ such that $\gr^F_p\DR(\M^\bullet)\neq0$. Indeed, when $\M\in MHM(X)$ and $X$ is smooth, it is immediate from the definition of the graded de Rham functor that the index $k$ of the first nonzero term in the Hodge filtration of $\M$ is the minimal $p$ such that $F_p\M\neq 0$, and in such case we have
$$
F_k\M\isom \gr^F_k\DR(\M)
$$
in $D^b_{coh}(X,\O_X)$. Since the graded de Rham functor is independent of the choice of an embedding of $X$, the \textit{first nonzero term in the Hodge filtration} $F_k\M$ of a mixed Hodge module $\M\in MHM(X)$ is well-defined even when $X$ is an embeddable singular variety. The following lemma (most likely known to experts) explains that under the direct image functors of the derived category of mixed Hodge modules, the index of the first nonzero term in the Hodge filtration does not decrease.

\begin{lem}
\label{lem: first HF}
    Let $f:X\to Y$ be a morphism of embeddable varieties and let $\M^\bullet\in D^bMHM(X)$ be an object in the derived category of mixed Hodge modules on $X$. Suppose
    $$
    \gr^F_p\DR(\M^\bullet)=0, \quad \forall p<k,
    $$
    for some integer $k$. Then,
    $$
    \gr^F_p\DR(f_*\M^\bullet)=0 \;\; \mathrm{and} \;\; \gr^F_p\DR(f_!\M^\bullet)=0, \quad \forall p<k.
    $$
\end{lem}

\begin{proof}
When $f$ is proper, this is immediate from Proposition \ref{prop: commutativity of graded de Rham}, that
$$
\gr^F_p\DR(f_*\M^\bullet)\isom Rf_*\left(\gr^F_p\DR(\M^\bullet)\right),
$$
and from the identity $f_!=f_*$ \cite{Saito90}*{4.3.3}. Notice that every morphism is a composition of an open embedding and a proper morphism. Hence, it suffices to prove this lemma when $f$ is an open embedding.

We first prove this for a single mixed Hodge module $\M\in MHM(X)$. By blowing up, we may reduce to an open embedding $f:X\hookrightarrow V$ where $X$ is the complement of a Cartier divisor $D\subset V$ and $V$ is embeddable. Since the statements are local, we can assume that $D$ is defined by a function $h:V\to \C$ and $V$ is smooth. Consider an embedding $(id,g):V\hookrightarrow V\times \C$. Then it suffices to prove when $\M$ is a mixed Hodge module on $V\times \C^*$ and $f:V\times \C^*\hookrightarrow V\times \C$. Notice that in this case, both $j_*\M$ and $j_!\M$ are mixed Hodge modules on $V\times \C$ and their Hodge filtrations $F$ are determined via their Kashiwara-Malgrange $V$-filtrations along $V\times \left\{0\right\}$. Hence, if $F_p\M=0$ for all $p<k$, then $F_p(j_*\M)=F_p(j_!\M)=0$ for all $p<k$ by \cite{Saito88}*{Proposition 3.2.2, Remarque 3.2.3, Lemma 5.1.17}. This proves the lemma for a mixed Hodge module $\M$.

Now for a nonzero object $\M^\bullet\in D^bMHM(X)$, we proceed by induction on the number of nonzero cohomologies of $\M^\bullet$. Suppose $\H^q(\M^\bullet)\neq 0$ and $\H^{>q}(\M^\bullet)=0$. Consider the following distinguished triangle
\begin{equation}
\label{eqn: truncation}
    \tau_{<q}\M^\bullet\to \M^\bullet\to \H^q(\M^\bullet)[-q]\xrightarrow{+1}.
\end{equation}
The cohomologies of $\tau_{<q}\M^\bullet$ are supported on degrees less than $q$, so the cohomologies of $\gr^F_p\DR(\tau_{<q}\M^\bullet)$ are also supported on degrees less than $q$. Hence, from the distinguished triangle
\begin{equation}
\label{eqn: truncation2}
\gr^F_p\DR(\tau_{<q}\M^\bullet)\to \gr^F_p\DR(\M^\bullet)\to \gr^F_p\DR(\H^q(\M^\bullet))[-q]\xrightarrow{+1},
\end{equation}
we have $\H^0(\gr^F_p\DR(\H^q(\M^\bullet)))=0$ for all $p<k$. Therefore, the index of the first nonzero term in the Hodge filtration of $\H^q(\M^\bullet)$ is at least $k$, from which we have
$$
\gr^F_p\DR(\tau_{<q}\M^\bullet)=\gr^F_p\DR(\H^q(\M^\bullet))=0$$
for all $p<k$. Consequently, we take the direct images of \eqref{eqn: truncation} and apply the induction hypothesis on $\tau_{<q}\M^\bullet$ and $\H^q(\M^\bullet)$ to obtain the conclusion.
\end{proof}

Additionally, given an object in $D^bMHM(X)$, we compare its graded de Rham complex with that of its first nonzero cohomology. This will later allow us to use a single mixed Hodge module to understand the graded de Rham complex of an object in $D^bMHM(X)$.

\begin{lem}
\label{lem: graded de Rham of first cohomology}
Let $X$ be an embeddable variety and $\M^\bullet\in D^bMHM(X)$ such that $\H^0(\M^\bullet)=\N\in MHM(X)$ and all the cohomologies in negative degrees are zero, that is $\H^{<0}(\M^\bullet)=0$. Suppose
$$
\gr^F_p\DR(\M^\bullet)=0, \quad \forall p<k,
$$
for some integer $k$. Then $\gr^F_{p}\DR\left(\N\right)=0$ for all $p<k$, and there exists an isomorphism
$$
\H^{k-p}\left(\gr^F_{p}\DR\left(\M^\bullet\right)\right)\isom \H^{k-p}\left(\gr^F_{p}\DR\left(\N\right)\right).
$$
for all $p\ge k$.
\end{lem}

\begin{proof}
We proceed by induction on the largest number $q$ such that $\H^q(\M^\bullet)\neq 0$. If there is no such $q$ or $q=0$, the statement is trivial. Suppose $q>0$ and consider the distinguished triangles \eqref{eqn: truncation} and \eqref{eqn: truncation2}.

Since the cohomologies of $\tau_{<q}\M^\bullet$ are supported on degrees less than $q$, it is easy to check that the cohomologies of $\gr^F_p\DR\left(\tau_{<q}\M^\bullet\right)$ are supported on degrees less than $q$.
Therefore, we have a surjection
$$
\H^q\left(\gr^F_{p}\DR\left(\M^\bullet\right)\right)\to \H^0\left(\gr^F_{p}\DR\left(\H^q(\M^\bullet)\right)\right)
$$
for all $p$. This implies that $\H^0\left(\gr^F_{p}\DR\left(\H^q(\M^\bullet)\right)\right)=0$ for all $p<k$, and thus $\gr^F_{p}\DR\left(\H^q(\M^\bullet)\right)=0$ for all $p<k$. Indeed, considered as a filtered $\D_\W$-module for an embedding $X\subset \W$, the index of the first nonzero term in the Hodge filtration of $\H^q(\M^\bullet)$ is at least $k$. Additionally, the cohomologies of $\gr^F_{p}\DR\left(\H^q(\M^\bullet)\right)$ are supported on degrees in the interval $[k-p, 0]$ (see \eqref{eqn: filtration on DR} or \cite{Saito88}*{Lemme 2.1.6}). Since $q$ is positive, we deduce that
$$
\gr^F_p\DR\left(\tau_{<q}\M^\bullet\right)=0,\quad \forall p<k,
$$
$$
\H^{k-p}\left(\gr^F_{p}\DR\left(\M^\bullet\right)\right)\isom \H^{k-p}\left(\gr^F_{p}\DR\left(\tau_{<q}\M^\bullet\right)\right).
$$
Therefore, we conclude by induction.
\end{proof}

Lastly, we state an analogue of \cite{KS21}*{Proposition 4.7} for objects in $D^bMHM(X)$. Recall that $\M^\bullet\in D^bMHM(X)$ has the underlying $\C$-complex $\DR(\M^\bullet)\in D^b_c(X,\C)$ in the bounded derived category of $\C$-sheaves on $X$ with constructible cohomologies. Notice that \eqref{eqn: filtration on DR} induces a filtration on this object, with respect to an embedding.

\begin{lem}
\label{lem: top HF}
    Let $X\subset \W$ be an embedding and $\M^\bullet \in D^bMHM(X)$. Let $F$ be the filtration on $\DR_\W (\M^\bullet)\in D^b_c(\W,\C)$, induced by \eqref{eqn: filtration on DR}. If $\gr^F_p\DR(\M^\bullet)=0$ for all $p>k$, then the inclusion
    $$
    F_k\DR_\W(\M^\bullet)\to \DR_\W(\M^\bullet)
    $$
    is a quasi-isomorphism.
\end{lem}

This follows immediately from the fact that the Hodge filtration $F$ is exhaustive for each component of the complex $\M^\bullet$. Alternatively, one may apply \cite{KS21}*{Proposition 4.7} for each component of $\M^\bullet$ to find some large $k'$ such that $F_{k'}\DR_\W(\M^\bullet)\to \DR_\W(\M^\bullet)$ is a quasi-isomorphism, from which the lemma follows easily.

\section{A description of rational singularities via Hodge modules}
\label{sec: rational singularities}

In this section, we characterize rational singularities using the polarizable Hodge modules associated to the intersection complexes. We first investigate a mixed Hodge module theoretic interpretation of Du Bois singularities from the literature. Recall from Saito \cite{Saito00} (or Theorem \ref{thm: Du Bois via MHM}) that $X$ has Du Bois singularities if and only if the natural morphism
$$
\O_X\to\gr^F_0\DR(\Q^H_X)
$$
in $D^b_{coh}(X,\O_X)$ is a quasi-isomorphism. The following rather elementary lemma simplifies this criterion.

\begin{lem}
\label{lem: finite rank 1}
    Let $\mu:X'\to X$ be a finite morphism of varieties. If there exists an isomorphism $\mu_*\O_{X'}\isom \O_X$ as a sheaf of $\O_X$-modules, then $\mu$ is an isomorphism.
\end{lem}

\begin{proof}
It suffices to prove that the natural morphism $\phi:\O_X\to \mu_*\O_{X'}$ of the sheaf of $\O_X$-algebras is an isomorphism. For each point $x\in X$, the restriction map $\phi|_{x}:\O_X|_{x}\to \mu_*\O_{X'}|_{x}$ is induced by the morphism $\mu|_{\mu^*(x)}:\mu^*(x)\to x$ where $\mu^*(x)$ is the scheme-theoretic inverse image of $x$. This restriction morphism is an isomorphism because the global section of $\mu^*(x)$ is a vector space of dimension $1$ over the residue field of $x$, from the assumption $\mu_*\O_{X'}\isom \O_X$. Hence, $\phi$ is an isomorphism.
\end{proof}

As a consequence, we conclude that $X$ has Du Bois singularities if and only if there exists some isomorphism
$$
\O_X\isom\H^0\left(\gr^F_0\DR(\Q^H_X)\right),
$$
and
$$
\H^k\left(\gr^F_0\DR(\Q^H_X)\right)=0, \quad \forall k\neq 0.
$$
Notice that the first condition assures that the weak normalization $\mu^{wn}:X^{wn}\to X$ is an isomorphism by the above lemma and \cite{Saito00}*{Proposition 5.2}; the second condition assures the vanishing of higher cohomologies of the Du Bois complex $\DB_X^0$. Our goal is to establish an analogous statement for rational singularities using polarizable Hodge modules. The following properties of intersection complexes are crucial for this purpose.

\begin{lem}
\label{lem: basic IC}
    Let $X$ be an embeddable irreducible variety and $\mu:\widetilde X\to X$ be a resolution of singularities. Then there exists a natural morphism
    $$
    \Q_X^H\to \IC^H_X[-\dim X]
    $$
    and quasi-isomorphisms
    $$
    R\mu_*\O_{\widetilde X}\isom\gr^F_0\DR\left(\IC^H_X[-\dim X]\right), \quad \mu_*\omega_{\widetilde X}\isom \gr^F_{-\dim X}\DR\left(\IC^H_X\right).
    $$
    Furthermore,
    $$
    \gr^F_{p}\DR\left(\IC^H_X\right)\isom 0,\quad \forall p>0 \;\mathrm{or}\; p<-\dim X.
    $$
\end{lem}

\begin{proof}
The first natural morphism is \cite{Saito90}*{4.5.11}. For the quasi-isomorphisms, notice that the direct image $\mu_*\Q^H_{\widetilde X}$ is generically the trivial variation of Hodge structures. Hence by Saito's Decomposition Theorem, we have
$$
\mu_*\Q^H_{\widetilde X}\isom \IC^H_X[-\dim X]\oplus \N^\bullet
$$
in $D^bMHM(X)$ where $\N^\bullet$ is supported on a proper subvariety of $X$. Let $\dim X=n$ for convenience. From Proposition \ref{prop: commutativity of graded de Rham}, we have
$$
R\mu_*\left(\gr^F_{-n}\DR\left(\Q^H_{\widetilde X}\right)\right)\isom \gr^F_{-n}\DR\left(\IC^H_X[-\dim X]\right)\oplus \gr^F_{-n}\DR (\N^\bullet).
$$
Together with $\gr^F_{-n}\DR\left(\Q^H_{\widetilde X}\right)\isom\omega_{\widetilde X}[-\dim X]$ and the Grauert–Riemenschneider vanishing theorem, the above quasi-isomorphism reduces to
$$
\mu_*\omega_{\widetilde X}[-\dim X]\isom\gr^F_{-n}\DR\left(\IC^H_X[-\dim X]\right)
$$
because $\mu_*\omega_{\widetilde X}$ is torsion free. Applying the duality \eqref{eqn: duality for IC}, we obtain
$$
R\Hom_{\O_X}\left(\mu_*\omega_{\widetilde X},\omega_X^\bullet\right)\isom \gr^F_0\DR\left(\IC^H_X\right),
$$
from which we deduce the lemma using Grothendieck duality on $\mu$. Lastly, if $p$ is not in the interval $[-\dim X,0]$, then $\gr^F_{p}\DR\left(\Q^H_{\widetilde X}\right)\isom 0$, which proves the last statement in the lemma by Saito's Decomposition Theorem.
\end{proof}

When $X$ is not necessarily an irreducible variety, there exists a natural morphism $\Q^H_X\to \Q^H_{X_i}$ in $D^bMHM(X)$ for each irreducible component $X_i$ of $X$. Taking a composition with the morphism in Lemma \ref{lem: basic IC}, we have a natural morphism $\Q^H_X\to \IC^H_{X_i}[-\dim X_i]$, that induces a natural morphism
\begin{equation}
\label{eqn: Q to IC}
\Q^H_X\to \bigoplus_i\IC^H_{X_i}[-\dim X_i].
\end{equation}

\begin{prop}
\label{prop: criterion for rational}
    Let $X=\cup_i X_i$ be a variety, where $X_i$ are the irreducible components of $X$. Then $X$ has rational singularities if there exists an isomorphism
    $$
    \O_X\isom \H^0\left(\gr^F_0\DR\left(\bigoplus_i\IC^H_{X_i}[-\dim X_i]\right)\right)
    $$
    and
    $$
   \H^k\left(\gr^F_0\DR\left(\bigoplus_i\IC^H_{X_i}[-\dim X_i]\right)\right)=0, \quad \forall k\neq 0.
    $$
    Conversely, if $X$ has rational singularities, then the following composition is a quasi-isomorphism:
    $$
    \O_X\to \gr^F_0\DR\left(\Q^H_X\right)\to \gr^F_0\DR\left(\bigoplus_i\IC^H_{X_i}[-\dim X_i]\right)
    $$
\end{prop}

\begin{proof}
Let $\mu^n:X^n\to X$ and $\mu^n_i:X^n_i\to X_i$ be normalizations. Notice that $X^n$ is a disjoint union of $X_i^n$. The first isomorphism is equivalent to $\O_X\isom \mu^n_*\O_{X^n}$, since we have
$$
\H^0\left(\gr^F_0\DR\left(\IC^H_{X_i}[-\dim X_i]\right)\right)\isom (\mu^n_i)_*\O_{X_i^n}
$$
from Lemma \ref{lem: basic IC} for each $X_i$. Consequently, the first isomorphism implies that $X$ is normal by Lemma \ref{lem: finite rank 1}. Suppose morphisms $\mu_i:\widetilde X_i\to X_i$ are resolutions of singularities. Then, $X$ has rational singularities if and only if $R^k(\mu_i)_*\O_{\widetilde X_i}=0$ for all $k\neq0$, which is equivalent to
$$
   \H^k\left(\gr^F_0\DR\left(\bigoplus_i\IC^H_{X_i}[-\dim X_i]\right)\right)=0, \quad \forall k\neq 0
$$
by Lemma \ref{lem: basic IC}.

For the converse statement, if $X$ has rational singularities, then $X$ is normal. Thus, it suffices to prove when $X$ is an irreducible normal variety with rational singularities. By Lemma \ref{lem: basic IC}, we have a quasi-isomorphism
$$
\O_X\isom \gr^F_0\DR\left(\IC^H_X[-\dim X]\right)
$$
and the composition
$$
\O_X\to \gr^F_0\DR\left(\Q^H_X\right)\to \gr^F_0\DR\left(\IC^H_{X}[-\dim X]\right)
$$
is a quasi-isomorphism away from the singular locus, which has a codimension at least two. Therefore, this composition is a quasi-isomorphism.
\end{proof}

\section{A characterization of Du Bois complexes of pairs}
\label{sec: characterization of Du Bois complexes}

For a pair $(X,Z)$ consisting of a variety $X$ and a closed subvariety $Z$, the Du Bois complex $\DB^0_{X,Z}$ does not generally have a simple interpretation through a resolution of singularities unless we assume certain singularity conditions on the complement. When $X\smallsetminus Z$ is smooth and a log resolution $\mu:(\widetilde X,E)\to (X,Z)$ is an isomorphism over $X\smallsetminus Z$, we have a simple characterization $\DB^0_{X,Z}\isom R\mu_*\O_{\widetilde X}(-E)$ obtained from the work of Du Bois \cite{DuBois81}. In Theorem \ref{thm: Du Bois complex, rational singularities}, we claim that this continues to hold when $X\smallsetminus Z$ has rational singularities and $\mu$ is not an isomorphism over $X\smallsetminus Z$. Our goal is to prove this theorem by utilizing the theory of mixed Hodge modules discussed in the previous sections.

\subsection{Proof of the characterization Theorem \ref{thm: Du Bois complex, rational singularities}}

We first prove the second quasi-isomorphism without the assumption of the singularity of the complement.

\begin{lem}
\label{lem: graded de Rham of j_!IC}
Let $X$ be an embeddable variety and $Z\subset X$ be a closed subvariety. Let $\mu:(\widetilde X,E)\to (X,Z)$ be a log resolution. Then there exists a quasi-isomorphism
$$
R\mu_*\O_{\widetilde X}(-E)\isom \gr^F_0\DR\left(j_!\left(\bigoplus_i\IC_{X_i\smallsetminus Z}^H[-\dim X_i]\right)\right).
$$
\end{lem}

\begin{proof}
Consider the following commuting diagram
\begin{equation*}
\xymatrix{
{E_i}\ar[r]\ar[d]& {\widetilde X_i}\ar[d]_-{\mu_i}& {\widetilde X_i\smallsetminus E_i}\ar[l]_-{j_i}\ar[d]_{\mu_i}\\
{Z\cap X_i}\ar[r]&{X_i}&{X_i\smallsetminus Z}\ar[l]_-{j_i}
}
\end{equation*}
where $\mu_i:(\widetilde X_i,E_i)\to (X_i,Z\cap X_i)$ is a log resolution of singularities for each irreducible component $X_i\subset \overline{X\smallsetminus Z}$ so that $\mu:(\widetilde X,E)\to (X,Z)$ is the disjoint union of $\widetilde X_i$ mapping to $X$. From Saito's Decomposition Theorem, we have
$$
(\mu_i)_*\Q^H_{\widetilde X_i\smallsetminus E_i}\isom \IC^H_{X_i\smallsetminus Z}[-\dim X_i]\oplus \N^\bullet,
$$
and by Lemma \ref{lem: basic IC}, we obtain $\gr^F_p\DR(\N^\bullet)=0$ for all $p\le-\dim X_i$. Applying Lemma \ref{lem: first HF} to $\N^\bullet$ and the locally closed morphism $j_i:X_i\smallsetminus Z\to X$, we have
$$
\gr^F_{-\dim X_i}\DR\left((j_i\circ\mu_i)_*\Q^H_{\widetilde X_i\smallsetminus E_i}\right)\isom \gr^F_{-\dim X_i}\DR\left((j_i)_*\IC^H_{X_i\smallsetminus Z}[-\dim X_i]\right),
$$
and using Proposition \ref{prop: commutativity of graded de Rham} to $\mu_i$, we have
\begin{align*}
\gr^F_{-\dim X_i}\DR\left((\mu_i\circ j_i)_*\Q^H_{\widetilde X_i\smallsetminus E_i}\right)&\isom R(\mu_i)_*\left(\gr^F_{-\dim X_i}\DR\left((j_i)_*\Q^H_{\widetilde X_i\smallsetminus E_i}\right)\right)\\
&\isom R(\mu_i)_*\left(\omega_{\widetilde X_i}(E_i)[-\dim X_i]\right).
\end{align*}
The last line follows from the fact that the first nonzero term in the Hodge filtration of the mixed Hodge module $(j_i)_*\Q^H_{\widetilde X_i\smallsetminus E_i}[\dim X_i]$ is $\omega_{\widetilde X_i}(E_i)$ (see, for instance, \cite{Saito90}*{Section 3.b}). Combining the above, we have
$$
\bigoplus_i\gr^F_{-\dim X_i}\DR\left(j_*\left(\IC^H_{X_i\smallsetminus Z}[\dim X_i]\right)\right)\isom\bigoplus_i R(\mu_i)_*\left(\omega_{\widetilde X_i}(E_i)[\dim X_i]\right).
$$
Take the duality $R\Hom_{\O_X}(-, \omega_X^\bullet)$ on both sides. Then the term on the right is quasi-isomorphic to $R\mu_*\O_{\widetilde X}(-E)$ by Grothendieck duality for $\mu$. The term on the left is quasi-isomorphic to $\gr^F_0\DR\left(j_!\left(\bigoplus_i\IC_{X_i\smallsetminus Z}^H[-\dim X_i]\right)\right)$ by Lemma \ref{lem: duality} and the identity $\dual j_*=j_!\dual$ (\cite{Saito90}*{4.3.5}).
\end{proof}

Back to the proof of Theorem \ref{thm: Du Bois complex, rational singularities}, it suffices to prove the quasi-isomorphism
$$
\DB^0_{X,Z}\isom \gr^F_0\DR\left(j_!\left(\bigoplus_i\IC_{X_i\smallsetminus Z}^H[-\dim X_i]\right)\right)
$$
by Lemma \ref{lem: graded de Rham of j_!IC}. From Theorem \ref{thm: Du Bois via MHM}, we have $\gr^F_0\DR(j_!\Q^H_{X\smallsetminus Z})=\DB_{X,Z}^0$. Notice that we have
\begin{align*}
R\Hom_{\O_X}\left(\gr^F_0\DR\left(j_!\Q^H_{X\smallsetminus Z}\right),\omega^\bullet_X\right)&\isom \gr^F_0\DR\left(\dual\left(j_!\Q^H_{X\smallsetminus Z}\right)\right)\\&\isom\gr^F_0\DR\left(j_*\left(\dual\left(\Q^H_{X\smallsetminus Z}\right)\right)\right)
\end{align*}
from Lemma \ref{lem: duality} and the identity $\dual j_!=j_*\dual$ (\cite{Saito90}*{4.3.5}), which reduces to proving a quasi-isomorphism
$$
\gr^F_0\DR\left(j_*\left(\dual\left(\Q^H_{X\smallsetminus Z}\right)\right)\right)\isom \bigoplus_i\gr^F_{-\dim X_i}\DR\left(j_*\left(\IC^H_{X_i\smallsetminus Z}[\dim X_i]\right)\right).
$$

Since $X\smallsetminus Z$ has rational singularities (which implies Du Bois by \cite{Kovacs99}*{Theorem S} or \cite{Saito00}*{Theorem 5.4}), Proposition \ref{prop: criterion for rational} induces the quasi-isomorphism
$$
\gr^F_0\DR\left(\Q^H_{X\smallsetminus Z}\right)\isom \gr^F_0\DR\left(\bigoplus_i\IC^H_{X_i\smallsetminus Z}[-\dim X_i]\right)
$$
induced by \eqref{eqn: Q to IC}. Consider the mapping cone $C^\bullet\in D^bMHM(X\smallsetminus Z)$ of \eqref{eqn: Q to IC}:
$$
\Q^H_{X\smallsetminus Z}\to \bigoplus_i\IC^H_{X_i\smallsetminus Z}[-\dim X_i]\to C^\bullet\xrightarrow{+1},
$$
then $\gr^F_p\DR(C^\bullet)=0$ for all $p\ge0$ because of \cite{Saito00}*{Theorem 0.2} and Lemma \ref{lem: basic IC}. Applying the duality Lemma \ref{lem: duality} and Lemma \ref{lem: first HF} on $C^\bullet$, we have
\begin{align*}
\gr^F_0\DR\left(j_*\left(\dual\left(\Q^H_{X\smallsetminus Z}\right)\right)\right)&\isom \gr^F_0\DR\left(j_*\left(\bigoplus_i\IC^H_{X_i\smallsetminus Z}(\dim X_i)[\dim X_i]\right)\right)\\
&\isom \bigoplus_i\gr^F_{-\dim X_i}\DR\left(j_*\left(\IC^H_{X_i\smallsetminus Z}[\dim X_i]\right)\right).
\end{align*}

\begin{rmk}
\label{rmk: naturality}
    Notice that $\DB^0_{\widetilde X,E}=\O_{\widetilde X}(-E)$ since $\widetilde X$ and $E$ have Du Bois singularities. Therefore, one may expect that the quasi-isomorphism $\DB^0_{X,Z}\isom R\mu_*\O_{\widetilde X}(-E)$ is the natural morphism induced by the functoriality of the Du Bois complexes. This is expected to be true, but we do not claim this at the current stage, mainly because we do not know that the functoriality of Du Bois complexes is induced by that of the objects in $D^bMHM(X)$.
\end{rmk}

\subsection{Log rational pairs and inversion of adjunction}

As an immediate consequence of Theorem \ref{thm: Du Bois complex, rational singularities}, a Du Bois pair $(X,Z)$ exhibits particularly nice properties when $X\smallsetminus Z$ has rational singularities. In this case, we call $(X,Z)$ a \textit{log rational pair}, which can be characterized by Corollary \ref{cor: criterion for log rational pair}. After proving this, we give a short proof of the inversion of adjunction for rational and Du Bois singularities (Corollary \ref{cor: inversion of adjunction}) as a consequence.

\subsubsection{Proof of Corollary \ref{cor: criterion for log rational pair}}

For the converse statement, $X\smallsetminus Z$ has rational singularities by \cite{Kovacs00}*{Theorem 1} and $(X,Z)$ is a Du Bois pair by \cite{Kovacs11}*{Theorem 5.4}. Hence, it remains to prove the first statement.

Suppose $(X,Z)$ is a log rational pair. By Theorem \ref{thm: Du Bois complex, rational singularities}, all higher cohomologies of $R\mu_*\O_{\widetilde X}(-E)$ are vanishing, and $\I_{X,Z}\isom\mu_*\O_{\widetilde X}(-E)$ as $\O_X$-modules, which is not necessarily a natural morphism. It suffices to prove that the natural morphism
$$
\I_{X,Z}\to \mu_*\O_{\widetilde X}(-E)
$$
is an isomorphism. Let $\mu^n:X^n\to X$ be the disjoint union of normalizations of the irreducible components $X_i$ of $\overline{X\smallsetminus Z}$, and let $Z^n=(\mu^n)^{-1}(Z)$. Notice that $\mu:\widetilde X\to X$ factors through $X^n$, and it is easy to see that the natural morphism
$$
\mu^n_*\left(\I_{X^n,Z^n}\right)\to \mu_*\O_{\widetilde X}(-E)
$$
is an isomorphism. Hence, $\mu^n_*\left(\I_{X^n,Z^n}\right)$ is isomorphic to $\I_{X,Z}$ as a $\O_X$-module. Now, consider the commutative diagram (see e.g. \cite{Kovacs11}*{Proposition 3.15}) induced by natural morphisms:
\begin{equation*}
\xymatrix{
{\I_{X,Z}}\ar[r]^-{\phi}\ar[d]& {\DB^0_{X,Z}}\ar[d]^-{\rho}\\
{\mu^n_*\left(\I_{X^n,Z^n}\right)}\ar[r]&{\mu^n_*\left(\DB^0_{X^n,Z^n}\right)}
}
\end{equation*}
From the assumption, $\phi$ is an isomorphism. Additionally, $\rho$ is an isomorphism since $\mu^n$ is an isomorphism over $X\smallsetminus Z$ (see \cite{DuBois81}*{Proposition 4.11}). Consequently, the natural morphism
$$
\I_{X,Z}\to \mu^n_*\left(\I_{X^n,Z^n}\right)
$$
admits a splitting $\mu^n_*\left(\I_{X^n,Z^n}\right)\isom \I_{X,Z}\oplus R$, where the support of $R$ does not contain irreducible components of $X$. Recall that $\mu^n_*\left(\I_{X^n,Z^n}\right)\isom \I_{X,Z}$ as $\O_X$-module, and $\I_{X,Z}$ has no embedded primes because $X$ is reduced. Therefore we conclude that $R=0$ and so the natural morphism $\I_{X,Z}\to \mu_*\O_{\widetilde X}(-E)$ is an isomorphism.

\subsubsection{Proof of Corollary \ref{cor: inversion of adjunction}}

It suffices to prove when $(X,D)$ is a log rational pair. Indeed, Kov\'acs \cite{Kovacs00b}*{Theorem 3.2} proved that $X$ has Du Bois singularities when $D$ has Du Bois singularities (see also \cite{KS11b}). This implies that $(X,D)$ is a log rational pair. Accordingly, by Corollary \ref{cor: criterion for log rational pair}, we have a natural quasi-isomorphism
$$
\O_X(-D)\to R\mu_*\O_{\widetilde X}(-E).
$$
Twisting by $\O_X(D)$, we have a natural quasi-isomorphism $\O_X\to R\mu_*\O_{\widetilde X}(\mu^*D-E)$. This factors through $R\mu_*\O_{\widetilde X}$, since $\mu^*D-E$ is effective, and hence we obtain a left quasi-inverse of $\O_X\to R\mu_*\O_{\widetilde X}$. Therefore $X$ has rational singularities \cite{Kovacs00}*{Theorem 1}.

\section{Results on extensions of forms}
\label{sec: extending forms}

\subsection{Extension criteria using Hodge modules}

We recall and modify certain techniques from Kebekus-Schnell \cite{KS21} to establish criteria for extending (log) forms. We start with the following simplified version of \cite{KS21}*{Proposition 6.4}, and include its proof for the reader's convenience. Given a complex of $\O_X$-modules with coherent cohomologies, this lemma provides a useful criterion to check the $S_2$ property of the first nonzero cohomology using the duality functor.

\begin{lem}
\label{lem: S2 criterion}
Let $X$ be an equidimensional variety of dimension $n$ and $K^\bullet\in D^b_{coh}(\O_X)$ be an object such that $\H^i(K^\bullet)=0$ for all $i<0$ and $\Supp\left(\H^0(K^\bullet)\right)=X$. If
$$
\dim \left(\Supp\left(R^i\Hom_{\O_X}\left(K^\bullet,\omega_X^\bullet\right)\right)\right)\le -(i+2), \quad \forall i\ge -n+1,
$$
then $\H^0(K^\bullet)$ satisfies $S_2$.
\end{lem}

\begin{proof}
It suffices to show that
$$
\dim \left(\Supp\left(R^i\Hom_{\O_X}\left(\H^0(K^\bullet),\omega_X^\bullet\right)\right)\right)\le -(i+2), \quad \forall i\ge -n+1,
$$
(see e.g. \cite{HL10}*{Proposition 1.1.6 (ii)}). Consider a distinguished triangle
$$
\H^0(K^\bullet)\to K^\bullet\to \tau_{\ge1}K^\bullet\xrightarrow{+1}
$$
where $\tau_{\ge1}K^\bullet$ is the truncation of $K^\bullet$ with cohomologies in degree $\ge 1$. After dualizing and taking cohomologies, we obtain the following exact sequence
$$
R^i\Hom_{\O_X}(K^\bullet,\omega_X^\bullet)\to R^i\Hom_{\O_X}(\H^0(K^\bullet),\omega_X^\bullet)\to R^{i+1}\Hom_{\O_X}(\tau_{\ge1}K^\bullet,\omega_X^\bullet).
$$
From the assumption, it suffices to prove
$$
\dim \left(\Supp\left(R^{i+1}\Hom_{\O_X}\left(\tau_{\ge 1}K^\bullet,\omega_X^\bullet\right)\right)\right)\le -(i+2), \quad \forall i\ge -n+1,
$$
which easily follows from \cite{HL10}*{Proposition 1.1.6 (i)} and the spectral sequence associated to the canonical filtration of $\tau_{\ge1}K^\bullet$.
\end{proof}

The next two lemmas build bridges between $p$-forms (with log poles) on a desingularization of a normal variety with the Hodge module associated to the intersection complex. Lemma \ref{lem: pushforward forms via IC} is \cite{KS21}*{Proposition 8.1}, for which we included a proof for the reader's convenience. Lemma \ref{lem: pushforward log forms via IC} generalizes \cite{KS21}*{Proposition 9.5} for the purpose of analyzing forms that extend with log poles over an arbitrary subvariety, not only over the singular locus. The main point is that the first nonzero cohomology of a certain graded de Rham complex describes the sheaf of forms that extend to forms on a resolution of singularities.

\begin{lem}
\label{lem: pushforward forms via IC}
Let $X$ be a normal variety of dimension $n$. Then the cohomologies of $\gr^F_{-p}\DR\left(\IC_X^H\right)$ are supported on degrees at least $-n+p$, and there is an isomorphism
$$
\mu_*\Omega^p_{\widetilde X}\isom \H^{-n+p}\left(\gr^F_{-p}\DR\left(\IC_X^H\right)\right)
$$
for a resolution of singularities $\mu:\widetilde X\to X$.
\end{lem}

\begin{proof}
The first statement is immediate from the definition of the graded de Rham functor and Lemma \ref{lem: basic IC}, which proves that the index of the first nonzero term in the Hodge filtration of $\IC^H_X$ is $-n$.

From Saito's Decomposition Theorem, we have
$$
\mu_*\Q^H_{\widetilde X}[n]\isom \IC^H_X\oplus \N^\bullet.
$$
By Proposition \ref{prop: commutativity of graded de Rham}, we have
$$
R\mu_*\left(\gr^F_{-p}\DR\left(\Q^H_{\widetilde X}[n]\right)\right)\isom \gr^F_{-p}\DR\left(\IC_X^H\right)\oplus \gr^F_{-p}\DR\left(\N^\bullet\right).
$$
Recall that $\gr^F_{-p}\DR\left(\Q^H_{\widetilde X}[n]\right)=\Omega^p_{\widetilde X}[n-p]$, from which we conclude that
$$
\mu_*\Omega^p_{\widetilde X}\isom \H^{-n+p}\left(\gr^F_{-p}\DR\left(\IC_X^H\right)\right)\oplus \H^{-n+p}\left(\gr^F_{-p}\DR\left(\N^\bullet\right)\right).
$$
Then the term $\H^{-n+p}\left(\gr^F_{-p}\DR\left(\N^\bullet\right)\right)$ of the direct sum is zero, because this is supported on a proper subvariety of $X$ while the left hand side is a torsion free sheaf.
\end{proof}

\begin{lem}
\label{lem: pushforward log forms via IC}
Let $(X,Z)$ be a pair consisting of a normal variety $X$ of dimension $n$ and a subvariety $Z$. Let $j:X\smallsetminus Z\to X$ be the open embedding and $\M_{X,Z}:=\H^0(j_*\IC^H_{X\smallsetminus Z})\in MHM(X)$. Then the cohomologies of $\gr^F_{-p}\DR\left(\M_{X,Z}\right)$ are supported on degrees at least $-n+p$, and there is an isomorphism
$$
\mu_*\Omega_{\widetilde X}^p(\log E)\isom \H^{-n+p}\left(\gr^F_{-p}\DR\left(\M_{X,Z}\right)\right)
$$
for a log resolution of singularities $\mu:(\widetilde X,E)\to (X,Z)$.
\end{lem}

\begin{proof}
Notice that the direct image functor $j_*$ of an open embedding is left exact with respect to the perverse structure. Hence, all the cohomologies of $j_*\IC^H_{X\smallsetminus Z}$ in negative degrees are zero. Applying Lemma \ref{lem: first HF}, Lemma \ref{lem: graded de Rham of first cohomology}, and Lemma \ref{lem: basic IC}, it is easy to see that the index of the first nonzero term in the Hodge filtration of $\M_{X,Z}$ is $-n$. Thus, the cohomologies of $\gr^F_{-p}\DR\left(\M_{X,Z}\right)$ are supported on degrees at least $-n+p$.

Consider the following commuting diagram:
\begin{equation*}
\xymatrix{
{E}\ar[r]\ar[d]& {\widetilde X}\ar[d]_-{\mu}& {\widetilde X\smallsetminus E}\ar[l]_-{j}\ar[d]_-{\mu}\\
{Z}\ar[r]&{X}&{X\smallsetminus Z}\ar[l]_-{j}
}
\end{equation*}
From Saito's Decomposition Theorem over $X\smallsetminus Z$, we have
$$
\mu_*\Q^H_{\widetilde X\smallsetminus E}[n]\isom \IC^H_{X\smallsetminus Z}\oplus \N^\bullet.
$$
Taking the direct image functor $j_*$ on both sides, we have an isomorphism
$$
\mu_*\left(j_*\Q^H_{\widetilde X\smallsetminus E}[n]\right)\isom j_*\IC^H_{X\smallsetminus Z}\oplus j_*\N^\bullet
$$
in $D^bMHM(X)$. Recall that $\gr^F_{-p}\DR\left(j_*\Q^H_{\widetilde X\smallsetminus E}[n]\right)=\Omega^p_{\widetilde X}(\log E)[n-p]$, from which we conclude
$$
\H^{-n+p}\left(\gr^F_{-p}\DR\left(j_*\IC^H_{X\smallsetminus Z}\right)\right)\isom\mu_*\Omega_{\widetilde X}^p(\log E) 
$$
as in the proof of Lemma \ref{lem: pushforward forms via IC}. Applying Lemma \ref{lem: graded de Rham of first cohomology}, we obtain
$$
\H^{-n+p}\left(\gr^F_{-p}\DR\left(j_*\IC^H_{X\smallsetminus Z}\right)\right)\isom\H^{-n+p}\left(\gr^F_{-p}\DR\left(\M_{X,Z}\right)\right),
$$
which completes the proof.
\end{proof}

Combining Lemma \ref{lem: pushforward forms via IC} or Lemma \ref{lem: pushforward log forms via IC} with Lemma \ref{lem: S2 criterion}, we have the following two propositions which provide a sufficient condition for the extension of (log) forms.

\begin{prop}
\label{prop: criterion for extending forms}
Let $X$ be a normal variety of dimension $n$. If we have
$$
\dim \left(\Supp\left(\H^{-p+k}\gr^F_{p-n}\DR\left(\IC^H_X\right)\right)\right)\le n-k-2, \quad \forall k\ge 1,
$$
then $\mu_*\Omega_{\widetilde X}^p$ is reflexive.
\end{prop}

\begin{proof}
By Lemma \ref{lem: pushforward forms via IC}, $\mu_*\Omega_{\widetilde X}^p$ is reflexive if and only if $\H^{-n+p}\left(\gr^F_{-p}\DR\left(\IC_X^H\right)\right)$ satisfies $S_2$. We apply Lemma \ref{lem: S2 criterion} to $\gr^F_{-p}\DR\left(\IC_X^H\right)[-n+p]$ with Lemma \ref{lem: duality}: if
$$
\dim \left(\Supp\left(\H^{i+n-p}\gr^F_{p-n}\DR\left(\IC^H_X\right)\right)\right)\le -(i+2), \quad \forall i\ge -n+1,
$$
then $\mu_*\Omega_{\widetilde X}^p$ is reflexive. This is equivalent to the proposition if $k$ is substituted by $i+n$.
\end{proof}

\begin{prop}
\label{prop: criterion for extending log forms}
In the setting of Lemma \ref{lem: pushforward log forms via IC}, let $\N_{X,Z}:=\dual \M_{X,Z}(-n)$. If we have
$$
\dim \left(\Supp\left(\H^{-p+k}\gr^F_{p-n}\DR\left(\N_{X,Z}\right)\right)\right)\le n-k-2, \quad \forall k\ge 1,
$$
then $\mu_*\Omega^p_{\widetilde X}(\log E)$ is reflexive.
\end{prop}

\begin{proof}
The proof of Proposition \ref{prop: criterion for extending forms} remains unchanged when replacing Lemma \ref{lem: pushforward forms via IC} with Lemma \ref{lem: pushforward log forms via IC}.
\end{proof}

\subsection{Proofs of the main theorems on extending forms}

We first establish certain lemmas used in the proof of Theorems \ref{thm: p-forms extend} and \ref{thm: p-log form extends}, and then proceed to prove these theorems, in the order: \ref{thm: p-forms extend}, \ref{thm: p-log form extends}, and \ref{thm: extending forms}.

One of the main ideas in \cite{KS21} is to cut a variety with general hyperplane sections and apply induction. Adapting this to our situation, we provide two versions describing the graded de Rham complexes after cutting; one is used for the extension of $p$-forms and the other is for the extension of $p$-forms with log poles.

\begin{lem}
\label{lem: slicing IC for forms}
Let $X$ be a normal variety, embeddable into a projective space $Y$. For a general hyperplane $L\subset Y$, there exists a short exact sequence of complexes
$$
0\to N^*_{L/Y}\otimes_{\O_L}\gr^F_{p+1}\DR\left(\IC_{X_L}^H\right)\to \O_L\tensor_{\O_Y}\gr^F_{p}\DR\left(\IC_{X}^H\right)\to \gr^F_{p}\DR\left(\IC_{X_L}^H\right)[1]\to 0
$$
where $X_L=L\cap X$ and $N^*_{L/Y}$ is the conormal bundle of $L\subset Y$.
\end{lem}

\begin{proof}
For a general hyperplane $L$, the embedding $\iota: L\hookrightarrow Y$ is non-characteristic for $\IC^H_X$. Therefore from Saito's theory, we have the non-characteristic pullback $\H^{-1}\left(\iota^*\IC^H_X\right)\isom \IC^H_{X_L}$ (see also \cite{Schnell16}*{Theorem 8.3}), and the statement in the lemma is equivalent to \cite{KS21}*{Proposition 4.17}.
\end{proof}

\begin{lem}
\label{lem: slicing IC for log forms}
In the setting of Lemma \ref{lem: pushforward log forms via IC}, suppose $X$ is embeddable into a projective space $Y$. Let $\N_{X,Z}:=\dual \M_{X,Z}(-n)$. For a general hyperplane $L\subset Y$, there exists a short exact sequence of complexes
$$
0\to N^*_{L/Y}\otimes_{\O_L}\gr^F_{p+1}\DR\left(\N_{X_L,Z_L}\right)\to \O_L\tensor_{\O_Y}\gr^F_{p}\DR\left(\N_{X,Z}\right)\to \gr^F_{p}\DR\left(\N_{X_L,Z_L}\right)[1]\to 0
$$
where $X_L=L\cap X$, $Z_L=L\cap Z$, and $N^*_{L/Y}$ is the conormal bundle of $L\subset Y$. Here, $\N_{X_L,Z_L}:=\dual \M_{X_L,Z_L}(-n+1)$ is a mixed Hodge module for the pair $(X_L,Z_L)$.
\end{lem}

\begin{proof}
Recall that the direct image functor $j_*$ of an open embedding is left exact with respect to the perverse structure. Hence, we have a distinguished triangle
$$
\M_{X,Z}\to j_*\IC^H_{X\smallsetminus Z}\to \tau_{\ge1} \left(j_*\IC^H_{X\smallsetminus Z}\right)\xrightarrow{+1}
$$
in $D^bMHM(X)$. Taking its dual and a Tate twist, we have
$$
\tau_{<0}\left(j_!\IC^H_{X\smallsetminus Z}\right)\to j_!\IC^H_{X\smallsetminus Z}\to \N_{X,Z}\xrightarrow{+1}
$$
using the identity $\dual j_*=j_!\dual$ (\cite{Saito90}*{4.3.5}) and the polarization $\IC^H_{X\smallsetminus Z}(n)\isom \dual\IC^H_{X\smallsetminus Z}$.

For a general hyperplane $L$, the embedding $\iota: L\hookrightarrow Y$ is non-characteristic for every nonzero cohomologies of $j_!\IC^H_{X\smallsetminus Z}$. From the identity $\iota^*j_!=j_!\iota^*$ (\cite{Saito90}*{4.4.3}), we have
$$
\iota^*\tau_{<0}\left(j_!\IC^H_{X\smallsetminus Z}\right)\to j_!\iota^*\IC^H_{X\smallsetminus Z}\to \iota^*\N_{X,Z}\xrightarrow{+1}
$$
Notice that from the general choice of $L$, we have $\iota^*\IC^H_{X\smallsetminus Z}\isom \IC_{X_L\smallsetminus Z_L}^H[1]$ and the cohomologies of the first term $\iota^*\tau_{<0}\left(j_!\IC^H_{X\smallsetminus Z}\right)$ are supported on degrees at most $-2$. Indeed, the pullback $\iota^*$ shifts the perverse structure by $[1]$. Therefore we have
$$
\H^0\left(j_!\IC_{X_L\smallsetminus Z_L}^H\right)\isom \H^{-1}(\iota^*\N_{X,Z})
$$
in $MHM(X)$. The left hand side is $\N_{X_L,Z_L}$ and the right hand side is the non-characteristic pullback of $\N_{X,Z}$. Then \cite{KS21}*{Proposition 4.17} completes the proof.
\end{proof}

From the fact that the dualizing functor preserves the perverse structure, we have $\H^0(j_!\IC^H_{X\smallsetminus Z})\isom \N_{X,Z}$ (see the proof of Lemma \ref{lem: slicing IC for log forms}). Using this, we have the following lemma, modified from \cite{KS21}*{Lemma 6.12}, crucial for the induction step in the proof of Theorem \ref{thm: p-forms extend} and Theorem \ref{thm: p-log form extends}.

\begin{lem}
\label{lem: initial step induction}
In the setting of Lemma \ref{lem: pushforward log forms via IC}, suppose $n\ge 2$. Let $\M\in MHM(X)$ be either $\IC^H_X$ or $\N_{X,Z}:=\dual \M_{X,Z}(-n)$. If $\H^i\left(\gr^F_0\DR(\M)\right)=0$ for $i=0,-1$, then $\H^0\left(\gr^F_{-1}\DR(\M)\right)=0$.
\end{lem}

\begin{proof}
Notice that $\H^0\left(\DR(\M)\right)=0$, considered as a cohomology of an object in $D_c^b(X,\C)$. Indeed, $\DR(\M)$ is a perverse sheaf admitting a nonzero morphism $\DR(\M)\to \H^0\left(\DR(\M)\right)$ of perverse sheaves. However, both $\DR(\IC^H_X)$ and $\DR(\N_{X,Z})$ do not admit a nonzero morphism to a vector space supported at a point, because their underlying $\C$-structures $\IC_X$ and $j_!\IC_{X\smallsetminus Z}$ do not.

Furthermore, $\gr^F_p\DR(\M)=0$ for all $p>0$. This follows from Lemma \ref{lem: basic IC} for $\IC^H_X$, Lemma \ref{lem: pushforward log forms via IC} and the duality Lemma \ref{lem: duality} for $\N_{X,Z}$. We take a local embedding $X\subset \W$ into a smooth variety and apply Lemma \ref{lem: top HF}. Then the rest follows from the proof of \cite{KS21}*{Lemma 6.8}. To elaborate, we have $\H^0\left(F_0\DR_\W(\M)\right)=0$ and from the short exact sequence of complexes
$$
0\to F_{-1}\DR_\W(\M)\to F_0\DR_\W(\M)\to \gr^F_0\DR(\M)\to 0,
$$
we have $\H^0(F_{-1}\DR_\W(\M))=0$. Therefore, again from the short exact sequence of complexes
$$
0\to F_{-2}\DR_\W(\M)\to F_{-1}\DR_\W(\M)\to \gr^F_{-1}\DR(\M)\to 0,
$$
we have $\H^0(\gr^F_{-1}\DR(\M))=0$. This is because $\H^1\left(F_{-2}\DR_\W(\M)\right)=0$ from the fact that $\DR_\W(\M)$ is supported on non-positive degrees.
\end{proof}

In the remaining section, we prove Theorems \ref{thm: p-forms extend}, \ref{thm: p-log form extends}, and \ref{thm: extending forms}. We modify the proofs of \cite{KS21}, using the properties of Du Bois complexes and intersection complexes developed earlier in the paper.

\subsubsection{Proof of Theorem \ref{thm: p-forms extend}}

Since the statements are local, we assume, without loss of generality, that $X$ is embeddable into a projective space $Y$. By Proposition \ref{prop: criterion for extending forms}, it suffices to prove
\begin{equation}
\label{eqn: S2 for p-form}
\dim \left(\Supp\left(\H^{-p+k}\gr^F_{p-n}\DR\left(\IC^H_X\right)\right)\right)\le n-k-2, \quad \forall k\ge 1
\end{equation}
for the ranges of $p$ given in the statement of the theorem. Notice that \eqref{eqn: S2 for p-form} is always true for $k>p$ since $\DR\left(\IC^H_X\right)$ is supported on non-positive degrees. We only consider $1\le k\le p$.

\textbf{Claim 1.} If $Z$ is empty, then \eqref{eqn: S2 for p-form} is true for all $p$.

This follows from \cite{KS21}*{Theorem 6.6}. We include the proof for completeness. We proceed by induction on $n=\dim X$. When $n=2$, we check for $p=1$ and $2$. For $p=2$, this is a consequence of Lemma \ref{lem: basic IC} since $X$ has rational singularities. Accordingly, $p=1$ follows from Lemma \ref{lem: initial step induction}.

Suppose \eqref{eqn: S2 for p-form} is true when the dimension of the variety is $n-1$. We prove for $\dim X=n$. Let $L$ be a general hyperplane of $Y$, then $X_L=L\cap X$ has rational singularities. By the induction hypothesis, the dimensions of the supports of both
$$
\H^{-p+k}\left(N^*_{L/Y}\otimes_{\O_L}\gr^F_{p-n+1}\DR\left(\IC_{X_L}^H\right)\right),\quad  \H^{-p+k}\left(\gr^F_{p-n}\DR\left(\IC_{X_L}^H\right)[1]\right)
$$
are at most $n-k-3$. By Lemma \ref{lem: slicing IC for forms}, the dimension of the support of $\O_L\tensor_{\O_Y}\H^{-p+k}\left(\gr^F_{p-n}\DR\left(\IC_{X}^H\right)\right)$ is at most $n-k-3$. We conclude that
$$
\dim \left(\Supp\left(\H^{-p+k}\gr^F_{p-n}\DR\left(\IC^H_X\right)\right)\right)\le \max\left\{n-k-2,0\right\}.
$$
Recall $k\le p$, so it suffices to consider when $(p,k)=(n-1,n-1),(n,n-1), (n,n)$. For $p=n$, \eqref{eqn: S2 for p-form} is a consequence of Lemma \ref{lem: basic IC} since $X$ has rational singularities. For $(p,k)=(n-1,n-1)$, this then follows from Lemma \ref{lem: initial step induction}.

\textbf{Claim 2.} If $p\le \codim_XZ-2$, then \eqref{eqn: S2 for p-form} is true.

We follow the proof of Claim 1. We proceed by induction on $n=\dim X$. When $n=2$ then either $Z$ is empty or $\codim_XZ=2$, in which case we already have \eqref{eqn: S2 for p-form}. For higher dimensions, we only need to prove when $Z$ is nonempty. Consider a general hyperplane $L\subset Y$, then $X_L\smallsetminus Z_L$ has rational singularities and either $\codim_{X_L}Z_L=\codim_XZ$ or $Z_L$ is empty. In both cases, it is easy to deduce \eqref{eqn: S2 for p-form} as in Claim 1.

\textbf{Claim 3.} If $R^{\codim_X Z-1}\mu_*\O_{\widetilde X}(-E)$ is generically vanishing on the components of $Z$ with the largest dimension, then \eqref{eqn: S2 for p-form} is true for $p\le \codim_XZ-1$.

We again proceed by induction on $n=\dim X$. Consider the distinguished triangle in the proof of Lemma \ref{lem: slicing IC for log forms}:
\begin{equation}
\label{eqn: truncation for j_!IC}
\tau_{<0}\left(j_!\IC^H_{X\smallsetminus Z}\right)\to j_!\IC^H_{X\smallsetminus Z}\to \N_{X,Z}\xrightarrow{+1}.
\end{equation}
From the 0-th cohomology of the natural morphism $j_!\IC^H_{X\smallsetminus Z}\to \IC_X^H$, we have a short exact sequence
\begin{equation}
\label{eqn: SES of N to IC}
0\to K\to \N_{X,Z}\to \IC_X^H\to 0 
\end{equation}
in $MHM(X)$, such that $K$ is the kernel of the surjection $\N_{X,Z}\to \IC_X^H$.

When $\dim X=2$ and $Z$ is a union of points, then by Lemma \ref{lem: graded de Rham of j_!IC} and the assumption, we have
$$
\H^i\left(\gr^F_0\DR\left(j_!\IC^H_{X\smallsetminus Z}\right)\right)=0
$$
for $i=0,-1$. Notice that $\H^0\left(\gr^F_0\DR\left(\tau_{<0}j_!\IC^H_{X\smallsetminus Z}\right)\right)=0$, since the cohomologies of $\tau_{<0}j_!\IC^H_{X\smallsetminus Z}$ are supported on negative degrees. From the long exact sequence of cohomologies after taking $\gr^F_0\DR$ on \eqref{eqn: truncation for j_!IC}, we have
$$
\H^i\left(\gr^F_0\DR(\N_{X,Z})\right)=0
$$ for $i=0,-1$. Accordingly, Lemma \ref{lem: initial step induction} deduces that $\H^0\left(\gr^F_{-1}\DR(\N_{X,Z})\right)=0$. Considering the long exact sequence of cohomologies of $\gr^F_{-1}\DR$ applied to  \eqref{eqn: SES of N to IC}, we conclude that $\H^0\left(\gr^F_{-1}\DR(\IC_X^H)\right)=0$, which proves the claim for $n=2$.

For $n>2$, assume $Z$ is nonempty. Otherwise, it follows from Claim 1. We again proceed by taking a general hyperplane $L\subset Y$. Since the higher direct images have a base change property for the restriction to a general hyperplane section, the assumption of the claim is true even after the restriction to $L$. Additionally, $X_L\smallsetminus Z_L$ has rational singularities, and either $\codim_{X_L}Z_L=\codim_XZ$ or $Z_L$ is empty. Therefore, from Lemma \ref{lem: slicing IC for forms} and the induction hypothesis, we conclude that
$$
\dim \left(\Supp\left(\H^{-p+k}\gr^F_{p-n}\DR\left(\IC^H_X\right)\right)\right)\le \max\left\{n-k-2,0\right\}.
$$
As in the proof of Claim 1, it suffices to consider when $(p,k)=(n-1,n-1),(n,n-1), (n,n)$. Since $p\le \codim_XZ-1$ and $Z$ is nonempty, $(p,k)=(n-1,n-1)$ is the only possible case. In this situation, $Z$ is a union of points. Then the proof for $n=2$ works verbatim.

This completes the proof of Theorem \ref{thm: p-forms extend}.

\subsubsection{Proof of Theorem \ref{thm: p-log form extends}}

This follows the same strategy as in the proof of Theorem \ref{thm: p-forms extend}. Under the assumptions of Theorem \ref{thm: p-log form extends}, we prove that
\begin{equation}
\label{eqn: S2 for log p-form}
\dim \left(\Supp\left(\H^{-p+k}\gr^F_{p-n}\DR\left(\N_{X,Z}\right)\right)\right)\le n-k-2, \quad \forall k\ge 1
\end{equation}
is true for $p\le m-2$. This is sufficient by Proposition \ref{prop: criterion for extending log forms}.

It is easy to see from Lemma \ref{lem: first HF} and the distinguished triangle \eqref{eqn: truncation for j_!IC} that $\gr^F_{p-n}\DR\left(\N_{X,Z}\right)=0$ for $p<0$. Additionally, \eqref{eqn: S2 for log p-form} is true for $k>p$ since $\DR\left(\N_{X,Z}\right)$ is supported on non-positive degrees. We only consider $1\le k\le p$.

We proceed by induction on $n=\dim X$. When $n=2$, the theorem is trivial unless $m\ge 3$. For $m\ge 3$, we have
\begin{equation}
\label{eqn: graded de Rham m=n+1}
\H^i\left(\gr^F_0\DR\left(j_!\IC^H_{X\smallsetminus Z}\right)\right)=0
\end{equation}
for $i=0,-1$ by Lemma \ref{lem: graded de Rham of j_!IC}. Then, we conclude as in Claim 3 of the proof of Theorem \ref{thm: p-forms extend}.

When $n>2$, as in the proof of Theorem \ref{thm: p-forms extend}, we use Lemma \ref{lem: slicing IC for log forms} and the induction hypothesis to conclude that
$$
\dim \left(\Supp\left(\H^{-p+k}\gr^F_{p-n}\DR\left(\N_{X,Z}\right)\right)\right)\le \max\left\{n-k-2,0\right\}.
$$
Thus, it suffices to consider when $(p,k)=(n-1,n-1),(n,n-1), (n,n)$. In all cases, we have $m\ge n+1$, which implies \eqref{eqn: graded de Rham m=n+1} for $i=0,-1$ by Lemma \ref{lem: graded de Rham of j_!IC}. Then, the proof for $n=2$ works in the same fashion.

\subsubsection{Proof of Theorem \ref{thm: extending forms}}

Consider the distinguished triangle
$$
\DB^0_{X,Z}\rightarrow\DB^0_X\xrightarrow{\rho}\DB^0_Z\xrightarrow{+1}
$$
and a quasi-isomorphism $\DB^0_{X,Z}\isom R\mu_*\O_{\widetilde X}(-E)$ induced by Theorem \ref{thm: Du Bois complex, rational singularities}. From the long exact sequence of cohomologies, we have
$$
\H^{i-1}\left(\DB^0_Z\right)\to R^i\mu_*\O_{\widetilde X}(-E)\to \H^{i}\left(\DB^0_X\right).
$$

It is a standard fact about Du Bois complexes that the dimension of the support of $\H^{i}\left(\DB^0_Z\right)$ is at most $\dim Z-i-1$ for $i>0$. Additionally, since $X$ has Du Bois singularities away from a subvariety of codimension $m$,
the support of $\H^i\left(\DB^0_X\right)$ for $i>0$ has codimension at least $m$. Therefore, when $i\ge1$, we have
$$
\codim_X\left(\Supp\left(R^i\mu_*\O_{\widetilde X}(-E)\right)\right)\ge\min\left\{m, \codim_X Z+i\right\}.
$$
Indeed if $i>1$, then this follows from the above. Suppose $i=1$. The morphism $\H^0\left(\DB^0_X\right)\to \H^0\left(\DB^0_Z\right)$ is equivalent to the natural morphism $\O_{X^{wn}}\to \O_{Z^{wn}}$ of the structure sheaves of weak normalizations. From the fact that $Z$ is generically smooth, hence generically weakly normal, this map is generically surjective, so that the support of the cokernel has the codimension at least $\codim_XZ+1$.

Therefore, if $m\ge \codim_XZ+1$, then $R^{\codim_X Z-1}\mu_*\O_{\widetilde X}(-E)$ is generically vanishing on the largest components of $Z$. Additionally,
$$
\max\left\{k\;|\;\codim_X\left(\Supp \left(R^{i-2}\mu_*\O_{\widetilde X}(-E)\right)\right)\ge i \mathrm{\;\;for\;all\;\;} 3\le i\le k\right\}\ge m.
$$
This completes the proof by Theorem \ref{thm: p-forms extend} and Theorem \ref{thm: p-log form extends}.

\section{Log canonical and Du Bois singularities}
\label{sec: log canonical and Du Bois}

In the rest of the paper, we show how to prove the theorem of Koll\'ar-Kov\'acs \cite{KK10}, that log canonical singularities are Du Bois, using the methods from Sections \ref{sec: Saito MHM} to \ref{sec: characterization of Du Bois complexes}. This new proof is rather simple and conceptual, beyond some necessary technicalities.

We start with a sketch of the general strategy to prove that $X$ is Du Bois when $X$ has log canonical singularities and $\w_X\isom \O_X$, in order to emphasize the main ideas in a technically simpler setting:
\begin{itemize}
\item Let $Z_0$ be a union of log canonical centers of $X$, and $Z_1$ be the union of log canonical centers properly contained in irreducible components of $Z_0$. One can reduce to proving that $(Z_0,Z_1)$ is a log rational pair.


\item Choose a log resolution $\mu:\widetilde X\to X$ such that $\mu^{-1}(Z_0)$ and $\mu^{-1}(Z_1)$ are reduced simple normal crossing (snc) divisors. Write $\mu^*K_X=K_{\widetilde X}+D_0+D_1+R-A$, so that the generic points of $D_0,D_1,R$ map to $Z_0\smallsetminus Z_1,Z_1, X\smallsetminus Z_0$ respectively, and  $D_0,D_1,R,A$ are effective snc divisors with $D_0,D_1,R$ reduced (because of log canonicity), not sharing irreducible components. Note that $D_0$ (resp. $D_1$) maps surjectively to $Z_0$ (resp. $Z_1$). Let $Y$ be the union of minimal strata of the snc divisor $D_0\cup R$ contained in $D_0$, so that $\mu|_Y:(Y,D_1|_Y)\to (Z_0,Z_1)$ is a proper surjective morphism of pairs.

We show a natural isomorphism $\mu_*\O_{D_0}(-D_1+A)\isom\I_{Z_0,Z_1}$, where $\I_{Z_0,Z_1}$ is the ideal sheaf of $Z_1$ in $Z_0$.
\item We realize the surjection $\O_{D_0}(-D_1+A)\to \O_Y(-D_1+A)$ as a morphism between the first nonzero terms in the Hodge filtrations of naturally defined mixed Hodge modules, the second of which is pure.
\item We obtain a left quasi-inverse map of the natural morphism $\I_{Z_0,Z_1}\to R\mu_*\O_Y(-D_1+A)$ using the Decomposition Theorem and the semisimplicity of a pure polarizable Hodge module. This implies that $(Z_0,Z_1)$ is a log rational pair.
\end{itemize}

In what follows, we describe additional technical backgrounds in Sections \ref{sec: associated primes of direct images} and \ref{sec: cyclic coverings}, which are used in the proof given in Section \ref{sec: lc implies Du Bois}.

\subsection{Associated primes of higher direct images of first nonzero terms in Hodge filtrations}
\label{sec: associated primes of direct images}

This subsection studies associated primes of higher direct images of the first nonzero terms in the Hodge filtrations of mixed Hodge modules. As previously mentioned in the paragraph above Lemma \ref{lem: first HF}, the first nonzero term in the Hodge filtration of a mixed Hodge module on $X$ is a well-defined coherent sheaf on $X$, regardless of the choice of an embedding into a smooth variety.

Let $X$ be a smooth variety and $\M\in MHM(X)$. Recall that the category of polarizable Hodge modules of weight $w$ admits the decomposition into the category of polarizable Hodge modules of weight $w$ with strict supports:
$$
MH(X,w)=\bigoplus_Z MH_Z(X,w),
$$
where $Z$ ranges over the set of irreducible subvarieties of $X$. In particular, each graded Hodge module $\gr^W_w\M$ admits the decomposition into polarizable Hodge modules of weight $w$ with strict supports. As an analogue of log canonical centers for log canonical pairs, we define centers for a mixed Hodge module $\M$.

\begin{defn}
A closed subvariety $Z\subset X$ is a \textit{center} of $\M\in MHM(X)$ if $Z$ is the strict support of a polarizable Hodge module appearing in the decomposition of $\gr^W_w\M$ for some $w$.
\end{defn}

This notion allows us to extend Koll\'ar's torsion freeness for higher direct images of canonical divisors \cite{Kollar86}*{Theorem 2.1} and Saito's extension to the first nonzero terms in the Hodge filtrations of polarizable Hodge modules \cite{Saito91}.

\begin{prop}
\label{prop: torsion freeness of direct image of MHM}
Let $f:X\to Y$ be a proper morphism from a smooth variety $X$ to a variety $Y$. Let $\M\in MHM(X)$ and $F_{p_0}\M$ be the first nonzero term in the Hodge filtration. Then for all $i\ge0$, every associated prime of $R^if_*(F_{p_0}\M)$ is the $f$-image of the generic point of a center of $\M$.
\end{prop}

\begin{ex}
When $\M=\Q^H_X[\dim X]$, then $p_0=-\dim X$ and $F_{p_0}\M=\omega_X$. Hence, we recover Koll\'ar's torsion freeness result; see \cites{Kollar86, Saito91}.
\end{ex}

\begin{ex}
When $\M=j_*\Q^H_U[\dim X]$ such that $j:U\to X$ is an open embedding of the complement of a reduced simple normal crossing divisor $E=X\smallsetminus U$, then $p_0=-\dim X$ and $F_{p_0}\M=\omega_X(E)$. It is easy to see from Lemma \ref{lem: decomposition snc} that centers of $\M$ are log canonical centers of the pair $(X,D)$. Consequently, Proposition \ref{prop: torsion freeness of direct image of MHM} says every associated prime of $R^if_*\omega_X(E)$ is the $f$-image of the generic point of a log canonical center of $(X,D)$; see \cite{Ambro03}*{Theorem 3.2} and \cite{Fujino11}*{Theorem 6.3}.
\end{ex}

In preparation for the proof, we first establish the commutativity result for higher direct images of the first nonzero term in the Hodge filtration:

\begin{lem}
\label{lem: commutativity of direct image and first HF}
In the setting of Proposition \ref{prop: torsion freeness of direct image of MHM}, assume $Y$ is smooth. Then $F_{p}\H^i(f_*\M)=0$ for all $p<p_0$ and there exists a natural isomorphism
$$
R^if_*F_{p_0}\M\isom F_{p_0}\H^i(f_*\M).
$$
This is functorial; for a morphism $\M_1\to \M_2$ of mixed Hodge modules $\M_1,\M_2\in MHM(X)$ with the induced morphism $F_{p_0}\M_1\to F_{p_0}\M_2$ of the first nonzero terms in the Hodge filtrations, there exists the following commutative diagram:
\begin{equation*}
\xymatrix{
{R^if_*F_{p_0}\M_1}\ar@{<->}[r]^-{\sim}\ar[d]& {F_{p_0}\H^i(f_*\M_1)}\ar[d]\\
{R^if_*F_{p_0}\M_2}\ar@{<->}[r]^-{\sim}&{F_{p_0}\H^i(f_*\M_2)}
}
\end{equation*}
\end{lem}

\begin{proof}
This is immediate from the definition of the direct image functor of filtered $\D$-modules. See, for example, \cite{Saito88}*{2.3}.
\end{proof}

\begin{proof}[Proof of Proposition \ref{prop: torsion freeness of direct image of MHM}]
Since the statements are local over $Y$, we may assume that $Y$ is smooth because $Y$ is locally embeddable. We proceed by induction on the number of weights $w$ of $\M$ with nonzero graded pieces $\gr^W_w \M$. If $\M$ is a polarizable Hodge module (i.e. pure weight), then the statement follows from Saito \cite{Saito91}*{Proposition 2.6}. If $\M$ has two or more weights with nonzero graded pieces, let $W_w\M$ be the first nonzero weight filtration of $\M$ and consider the short exact sequence
$$
0\to W_w\M\to \M\to \M/W_w\M\to0
$$
in $MHM(X)$. Then, we have a long exact sequence of cohomologies of direct images
$$
\to\H^i(f_*W_w\M)\xrightarrow{\rho_i} \H^i(f_*\M)\xrightarrow{\psi_i} \H^i(f_*\M/W_w\M)\to
$$
in $MHM(Y)$. Let $K_i=\mathrm{image}(\rho_i)$, $Q_i=\mathrm{image}(\psi_i)$. The strictness of the Hodge filtrations induces the short exact sequence
$$
0\to F_{p_0}K_i\to F_{p_0}\H^i(f_*\M)\to F_{p_0}Q_i\to 0.
$$
Notice that $\H^i(f_*W_w\M)$ is a polarizable Hodge module of weight $w+i$, and hence, semisimple (\cite{Saito88}*{Corollaire 5.2.13}). Consequently, the surjection
$$
\H^i(f_*W_w\M)\to K_i
$$
splits. Therefore, $F_{p_0}K_i$ is a subsheaf of $F_{p_0}\H^i(f_*W_w\M)$ and $F_{p_0}Q_i$ is a subsheaf of $F_{p_0}\H^i(f_*\M/W_w\M)$. Thus, every associated prime of $F_{p_0}\H^i(f_*\M)$ is an associated prime of either $F_{p_0}\H^i(f_*W_w\M)$ or $F_{p_0}\H^i(f_*\M/W_w\M)$. By Lemma \ref{lem: commutativity of direct image and first HF} and the induction hypothesis, we conclude the proof.
\end{proof}

\subsection{Mixed Hodge modules of cyclic coverings}
\label{sec: cyclic coverings}

This subsection studies various mixed Hodge modules appearing in cyclic covering constructions, branched along simple normal crossing divisors. Let $X$ be a smooth variety and $D+R+\Delta$ is a $\Q$-effective snc divisor such that $D+R$ is reduced, $\Delta$ has coefficients less than $1$, and $D,R,\Delta$ do not share an irreducible component. Denote $I_D,I_R,I_\Delta$ as the index sets of irreducible components of $D,R,\Delta$, respectively. Equivalently, $D=\cup_{i\in I_D}D_i$, $R=\cup_{i\in I_R}R_i$, $\lceil\Delta\rceil=\cup_{i\in I_\Delta}\Delta_i$. Let $L$ be a line bundle on $X$ such that $L\sim_\Q \Delta$, or equivalently, there exists a positive integer $m$ such that $L^{\tensor m}\isom \O_X(m\Delta)$.

Let $\pi:X'\to X$ be the cyclic cover (see e.g. \cite{Kollar13}*{2.44}) associated to the section $s:\O_X\to L^{\tensor m}$ inducing the divisor $m\Delta$. Let $D':=\pi^*D$, $R':=\pi^*R$. We describe the local picture of $X'$ as the normalization of a hypersurface of the total space of $L$.

Let $x_1,\dots, x_n$ be the local coordinate of $X$ and $\lceil\Delta\rceil=V(x_1\cdots x_a)$, $D=V(x_{a+1}\cdots x_b)$, $R=V(x_{b+1}\cdots x_c)$ where $1\le a\le b\le c\le n$. If $t$ is a local generator of $L$, then $X'$ is locally the normalization of a hypersurface
\begin{equation}
\label{eqn: cyclic cover hypersurface}
t^m=x_1^{e_1}\cdots x_a^{e_a}    
\end{equation}
in the coordinate system $x_1,\dots,x_n,t$. With this coordinate system, we have
$$
D'=V(x_{a+1}\cdots x_b),\quad R'=V(x_{b+1}\cdots x_c),
$$
that $D_i':=\pi^*D_i$ and $R_i':=\pi^*R_i$ are locally defined by the vanishing of a single coordinate. For a subset of indices $I\subset I_D\sqcup I_R$, we define the stratum
$$
(D\cup R)_I:=D_{I\cap I_D}\bigcap R_{I\cap I_R},
$$
where $D_{I\cap I_D}:=\cap_{i\in I\cap I_D}D_i$ and $R_{I\cap I_R}:=\cap_{i\in I\cap I_R}R_i$. Alternatively, this is a smooth stratum of the snc divisor $D\cup R$. Likewise, we define $(D'\cap R')_I$. It is easy to see that $(D'\cap R')_I=\pi^*(D\cap R)_I$, so that this is locally the normalization of the hypersurface \eqref{eqn: cyclic cover hypersurface} cut out by the corresponding coordinates in $x_{a+1},\dots, x_c$.

Notice that $X'$ and $(D'\cap R')_I$ have finite quotient singularities by \cite{EV92}*{Lemma 3.24}, hence are rationally smooth. This implies that
\begin{equation}
\label{eqn: rationally smooth}
\Q_{X'}^H[n]\isom \IC_{X'}^H\isom \left(\dual\Q_{X'}^H\right)(-n)[-n]    
\end{equation}
in $MH(X',n)$ and likewise for $(D'\cap R')_I$ (see also \cite{HTT}*{Proposition 8.2.21}). Indeed, the natural morphism $\Q^H_{X'}[n]\to \IC^H_{X'}$ in Lemma \ref{lem: basic IC} is an isomorphism, and $\IC_{X'}^H\isom \left(\dual\Q_{X'}^H\right)(-n)[-n]$ by the polarization.

For brevity, given a variety $Z$, we define $\mathbb D_Z^H:=\dual \Q_Z^H[-\dim Z]\in D^bMHM(Z)$ as the object associated to the Verdier's dualizing complex of $Z$ shifted by $[-\dim Z]$. Using this notation and \eqref{eqn: rationally smooth}, $\mathbb D_{X'}^H$ (resp. $\mathbb D_{(D'\cap R')_I}^H$) is a polarizable Hodge module of weight $-n$ (resp. $-n+|I|$) with strict support $X'$ (resp. $(D'\cap R')_I$). Furthermore, from the identity $\pi^!\dual=\dual\pi^*$ (\cite{Saito90}*{4.4}), it is easy to see that
\begin{equation}
\label{eqn: cyclic cover pullback}
\mathbb D_{X'}^H=\pi^!\mathbb D_{X}^H, \quad \mathbb D_{(D'\cap R')_I}^H=\pi^!\mathbb D_{(D\cap R)_I}^H.    
\end{equation}

\begin{lem}
\label{lem: decomposition snc}
With the notation above, $\mathbb D^H_D\in MHM(X)$ and we have the decomposition
$$
\gr^W_w\mathbb D^H_D=\bigoplus_{I\subset I_D, |I|=n+w} \mathbb D_{D_I}^H. 
$$
Furthermore, the first nonzero term in the Hodge filtration is $F_0 \mathbb D^H_D\isom \omega_D$.
\end{lem}

This is the statement that holds in general for any reduced snc divisor $D$ in $X$.

\begin{proof}
We proceed by induction on the number $|I_D|$ of indices of $D$. If $D$ is a smooth divisor, then the proof is immediate. If $D$ has two or more irreducible components, let $D_0$ be a component and write $D=D_0\cup D_1$ with $D_1$ not containing $D_0$.

Let $\alpha:D_0\to D$ be the closed embedding and $\beta:D\smallsetminus D_0\to D$ be the open embedding. Then we have the distinguished triangle (Proposition \ref{prop: Saito triangle} or \cite{Saito90}*{4.4.1})
\begin{equation}
\label{eqn: SES on snc1}
\mathbb D_{D_0}^H\to \mathbb D_D^H\to \beta_*\mathbb D^H_{D\smallsetminus D_0}\xrightarrow{+1}    
\end{equation}
which is a short exact sequence in $MHM(X)$. Indeed, $\beta$ is an affine open embedding, and hence $\beta_*$ preserves the perverse structure. Therefore by the induction hypothesis, $\beta_*\mathbb D_{D\smallsetminus D_0}\in MHM(X)$ and thus $\mathbb D_D^H\in MHM(X)$.

Since $D\smallsetminus D_0=D_1\smallsetminus D_0$, the open embedding $\beta$ factors through the open embedding $\beta_1:D\smallsetminus D_0\to D_1$. Using \cite{Saito90}*{4.4.1} again, there exists a distinguished triangle
\begin{equation}
\label{eqn: SES on snc2}
\mathbb D_{D_1}^H\to \beta_*\mathbb D^H_{D\smallsetminus D_0}\to \mathbb D^H_{D_0\cap D_1}\xrightarrow{+1}.
\end{equation}
This is valid because for the closed embedding $\alpha_1:D_0\cap D_1\to D_1$, the pullback $\alpha_1^!\mathbb D_{D_1}\isom \mathbb D_{D_0\cap D_1}[-1]$. Likewise, this is a short exact sequence in $MHM(X)$.

Since the weight filtration is strict for morphisms in the category of mixed Hodge modules, \eqref{eqn: SES on snc1} and \eqref{eqn: SES on snc2} remain to be short exact sequences in $MH(X,w)$ upon taking graded pieces $\gr^W_w$. Therefore, we apply the induction hypothesis on the snc divisors $D_1\subset X$ and $D_0\cap D_1\subset D_0$ to obtain the decomposition using the fact that $MH(X,w)$ is semisimple.

The first nonzero term in the Hodge filtration, $F_0 \mathbb D^H_D\isom \omega_D$, easily follows from the fact that $D$ has Du Bois singularities and the duality Lemma \ref{lem: duality} (see \cite{Saito00}*{Corollary 0.3}).
\end{proof}

Combining Lemma \ref{lem: decomposition snc} and \eqref{eqn: cyclic cover pullback}, it is easy to see that $\mathbb D^H_{D'}\in MHM(X')$ and
\begin{equation}
\label{eqn: decomposition cyclic cover}
\gr^W_w\mathbb D^H_{D'}=\bigoplus_{I\subset I_D, |I|=n+w}\mathbb D^H_{D'_I}.
\end{equation}
Consider the following diagram:
\begin{equation*}
\xymatrix{
{D'\smallsetminus R'}\ar[r]^-{\iota'}\ar[d]_-{j'}& {X'\smallsetminus R'}\ar[d]^-{j'}\\
{D'}\ar[r]^-{\iota'}&{X'}
}
\end{equation*}
where $\iota':D'\to X'$ is the closed embedding and $j':X'\smallsetminus R'\to X'$ is the open embedding. Similar to the previous lemma, we investigate the properties of the object
$$
\M':=j'_*\mathbb D_{D'\smallsetminus R'}^H\in D^bMHM(X').
$$

\begin{prop}
\label{prop: cyclic cover MHM}
With the notation in the previous paragraphs, $\M'\in MHM(X')$ and if $w\in [-n+1,0]$, then
$$
\gr^W_{w}\M'=\bigoplus_{I}\mathbb D^H_{(D'\cup R')_I} \quad \mathrm{for\;all\;\;} I\subset I_D\sqcup I_R \mathrm{\;\;with\;\;} I\nsubseteq I_R, |I|=n+w
$$
in $MH(X', w)$. If $w\notin [-n+1,0]$, then $\gr^W_{w}\M'=0$. Furthermore, the first nonzero term in the Hodge filtration of $\M'$ is
$$
F_{0}\M'\isom \omega_{D'}(R').
$$
\end{prop}

\begin{proof}
Notice that $\mathbb D^H_{D'\smallsetminus R'}\in MHM(X'\smallsetminus R')$ and $j'$ is an affine open embedding, which implies that $\M'\in MHM(X')$. We prove the rest of the statements.

Denote the open embeddings $j:X\smallsetminus R\to X$ and $j_0:X\smallsetminus (R\cup D)\to X\smallsetminus R$. As in the proof of Lemma \ref{lem: decomposition snc}, we have the distinguished triangle
$$
\mathbb D^H_{X}\to j_*\mathbb D^H_{X\smallsetminus R}\to \mathbb D^H_{R}\xrightarrow{+1}
$$
which is a short exact sequence in $MHM(X)$. Recall that $\pi^!j_*=j'_*\pi^!$ (\cite{Saito90}*{4.4.3}). Applying the functor $\pi^!$, we have the distinguished triangle
$$
\mathbb D^H_{X'}\to j'_*\mathbb D^H_{X'\smallsetminus R'}\to \mathbb D^H_{R'}\xrightarrow{+1}.
$$
Denote $j_0':X'\smallsetminus (R'\cup D')\to X'\smallsetminus R'$, then we likewise have
$$
\mathbb D^H_{X'}\to j'_*j'_{0*}\mathbb D^H_{X'\smallsetminus (R'\cup D')}\to \mathbb D^H_{R'\cup D'}\xrightarrow{+1}.
$$
In addition, we have a distinguished triangle (\cite{Saito90}*{4.4.1})
\begin{equation}
\label{eqn: SES for M'}
j'_*\mathbb D^H_{X'\smallsetminus R'}\to j'_*j'_{0*}\mathbb D^H_{X'\smallsetminus (R'\cup D')}\to \M'\xrightarrow{+1},
\end{equation}
since $\M'=j'_*(\iota')^!\mathbb D^H_{X'\smallsetminus R'}[1]$. Taking graded pieces $\gr^W_w$ of the previous three distinguished triangles and applying \eqref{eqn: decomposition cyclic cover} with the semisimplicity of $MH(X',w)$, we obtain the decomposition
$$
\gr^W_{w}\M'=\bigoplus_{I}\mathbb D^H_{(D'\cup R')_I} \quad \mathrm{for\;all\;\;} I\subset I_D\sqcup I_R \mathrm{\;\;with\;\;} I\nsubseteq I_R, |I|=n+w.
$$
In particular, $\gr^W_w\M'=0$ if $w\notin [-n+1,0]$.

Since $X'$ has rational singularities, $(X',R')$ and $(X',R'\cup D')$ are log rational pairs with ideal sheaves $\I_{X',R'}=\O_{X'}(-R')$ and $\I_{X',R'\cup D'}=\O_{X'}(-R'-D')$. Dualizing Theorem \ref{thm: Du Bois complex, rational singularities} and using \eqref{eqn: rationally smooth} with Lemma \ref{lem: duality}, we deduce that the first nonzero terms in the Hodge filtrations are
$$
F_0j'_*\mathbb D^H_{X'\smallsetminus R'}\isom\omega_{X'}(R'),\quad F_0j'_*j'_{0*}\mathbb D^H_{X'\smallsetminus (R'\cup D')}\isom\omega_{X'}(R'+D').
$$
Applying these to the first nonzero term in the Hodge filtration $F_0$ of \eqref{eqn: SES for M'}, we obtain a short exact sequence
$$
0\to \omega_{X'}(R')\to \omega_{X'}(R'+D')\to F_0\M'\to 0.
$$
Notice that $F_0\M'$ is the first nonzero term in the Hodge filtration of $\M'$. Over the snc locus of $(X',R'\cup D')$, the morphism $\omega_{X'}(R')\to \omega_{X'}(R'+D')$ is the natural inclusion. Since both $\omega_{X'}(R')$ and $\omega_{X'}(R'+D')$ are $S_2$-sheaves, this morphism should be the natural inclusion over all $X'$. Therefore,
$$
F_0\M'\isom \omega_{D'}(R')
$$
by the adjunction formula.
\end{proof}

\begin{cor}
\label{cor: cyclic cover MHM on X'}
In the setting of Proposition \ref{prop: cyclic cover MHM}, $j'_*\mathbb D^H_{X'\smallsetminus R'}\in MHM(X')$ and if $w\in [-n,0]$, then
$$
\gr^W_{w}j'_*\mathbb D^H_{X'\smallsetminus R'}=\bigoplus_{I}\mathbb D^H_{ R'_I} \quad \mathrm{for\;all\;\;} I\subset I_R \mathrm{\;\;with\;\;} |I|=n+w
$$
in $MH(X',w)$. If $w\notin [-n,0]$, then $\gr^W_wj'_*\mathbb D^H_{X'\smallsetminus R'}=0$. Furthermore, the first nonzero term in the Hodge filtration of $j'_*\mathbb D^H_{X'\smallsetminus R'}$ is
$$
F_{0}j'_*\mathbb D^H_{X'\smallsetminus R'}\isom \omega_{X'}(R').
$$
\end{cor}

Here, we denote $\mathbb D^H_{ R'_\emptyset}:=\mathbb D^H_{X'}$ by convention, so that $\gr^W_{-n}j'_*\mathbb D^H_{X'\smallsetminus R'}=\mathbb D^H_{X'}$.

\begin{proof}
This follows from the proof of Proposition \ref{prop: cyclic cover MHM}.
\end{proof}

\begin{rmk}
\label{rmk: eigensheaf decomposition}
Since $\pi:X'\to X$ and its restriction $\pi|_{D'}:D'\to D$ are degree $m$ cyclic covers, the pushforward $\pi_*F_0\M'\isom \pi_*\omega_{D'}(R')$ of the first nonzero term in the Hodge filtration admits a $\mu_m$-action (i.e. an action by the group $\mu_m$ of $m$-th roots of unity). More precisely, there exists a natural decomposition
$$
\pi_*\omega_{D'}(R')\isom \bigoplus_{i=0}^{m-1}\omega_D(R)\tensor L^{\tensor i}(-\lfloor i\Delta\rfloor)
$$
into $\mu_m$-eigensubsheaves. In fact, $R'=\pi^*R$, so this follows from the general fact on a cyclic covering of a demi-normal scheme (see \cite{Kollar13}*{2.44.5}).

Likewise, for the pushforward $\pi_*F_0j'_*\mathbb D^H_{X'\smallsetminus R'}\isom \pi_*\omega_{X'}(R')$, there exists a natural decomposition
$$
\pi_*\omega_{X'}(R')\isom \bigoplus_{i=0}^{m-1}\omega_X(R)\tensor L^{\tensor i}(-\lfloor i\Delta\rfloor)
$$
into $\mu_m$-eigensubsheaves.
\end{rmk}

\subsection{Log canonical singularities are Du Bois}
\label{sec: lc implies Du Bois}

Based on the machinery developed, we give an alternative proof of the theorem of Koll\'ar-Kov\'acs, that log canonical singularities are Du Bois.

\textbf{Setup.} Let $f:X\to Y$ be a projective morphism from an irreducible smooth variety $X$ to an irreducible normal variety $Y$. Let $\Delta$ be an snc $\Q$-divisor on $X$ with coefficients less than or equal to $1$. Define $\Delta^{=1}$ as the sum of the components of $\Delta$ with coefficients equal to $1$ and $\Delta^{<1}:=\Delta-\Delta^{=1}$. Then the strata of the reduced snc pair $(X,\Delta^{=1})$ are the log canonical centers of $(X,\Delta)$.

Let $Z_0\subset Y$ be the $f$-image of a union of log canonical centers of $(X,\Delta)$. Subsequently, let $Z_1\subsetneq Z_0$ be the $f$-image of the union of log canonical centers of $(X,\Delta)$ in $f^{-1}(Z_0)$, not mapping surjectively onto an irreducible component of $Z_0$. Notice that $Z_1$ may be empty.

\begin{thm}
\label{thm: lc admits log rational stratification}
With the setup above, if $K_X+\Delta\sim_{\Q,f}0$ and the natural morphism
$$
\O_Y\to f_*\O_X(\lceil-\Delta^{<1}\rceil)$$
is an isomorphism, then the pair $(Z_0,Z_1)$ is log rational.
\end{thm}

In the proof, we frequently use the following remark from Ambro \cite{Ambro11}*{Remark 4.1} which we state separately for convenience.

\begin{rmk}
\label{rmk: Ambro}
For a smooth variety $X$ and a reduced snc divisor $E=\cup_{i\in I}E_i$ with smooth irreducible components $E_i$, let $I'$ and $I''$ be nonempty subsets of $I$. Suppose $C$ is an irreducible component of $\cap_{i\in I'}E_i$. Then $C\subset \cup_{i\in I''}E_i$ if and only if $I'\cap I''\neq \emptyset$.
\end{rmk}

\begin{proof}[Proof of Theorem \ref{thm: lc admits log rational stratification}]
We take a log resolution of singularities $\mu:(\widetilde X,\widetilde \Delta) \to (X,\Delta)$ so that $\mu^*(K_X+\Delta)=K_{\widetilde X}+\widetilde \Delta$ and the $f\circ\mu$-inverse images $(f\circ \mu)^{-1}(Z_0)$ and $(f\circ \mu)^{-1}(Z_1)$ are reduced snc divisors. We may further assume that the union
$$
\Supp(\widetilde\Delta)\cup(f\circ \mu)^{-1}(Z_0)\cup(f\circ \mu)^{-1}(Z_1)
$$
is an snc divisor. It is easy to check that $K_{\widetilde X}+\widetilde \Delta\sim_{\Q,f\circ \mu}0$ and the natural morphism 
$$
\O_Y\to(f\circ\mu)_*\O_{\widetilde X}(\lceil-\widetilde\Delta^{<1}\rceil)
$$
is an isomorphism. Moreover, the $\mu$-images of log canonical centers of $(\widetilde X,\widetilde \Delta)$ are log canonical centers of $(X,\Delta)$. Hence, without loss of generality, we assume that $f^{-1}(Z_0)$ and $f^{-1}(Z_1)$ are reduced snc divisors such that
$$
\Supp(\Delta)\cup f^{-1}(Z_0)\cup f^{-1}(Z_1)
$$
is an snc divisor. Additionally, we assume $K_X+\Delta\sim_\Q 0$ and $Y$ is embeddable, since the statements are local on $Y$.

Let $D_0$ (resp. $D_1$) be the sum of divisors in both $\Delta^{=1}$ and $f^{-1}(Z_0\smallsetminus Z_1)$ (resp. $f^{-1}(Z_1)$), and let $R:=\Delta^{=1}-D_0-D_1$, which corresponds to the sum of divisors in $\Delta^{=1}$ not lying above $Z_0$. For convenience, write $\Delta^{<1}=\Delta^b-A$ where $A=\lceil-\Delta^{<1}\rceil$ is an effective integral divisor and $\Delta^b$ is an effective divisor with coefficients less than $1$. In summary,
$$
\Delta=D_0+D_1+R+\Delta^b-A.
$$

Let a line bundle $L:=-(K_X+D_0+D_1+R-A)$ so that $L\sim_\Q \Delta^b$. Subsequently, there exists a cyclic cover $\pi:X'\to X$ of degree $m$ associated to the isomorphism $L^{\tensor m}\isom \O_X(m\Delta^b)$.

Consider the mixed Hodge module $j'_*\mathbb D^H_{X'\smallsetminus R'}\in MHM(X')$ in Corollary \ref{cor: cyclic cover MHM on X'}. We have the exact sequence
$$
f_*\O_X(A-D_1)\to f_*\O_{D_0}(A-D_1)\to R^1f_*\O_X(A-D_0-D_1).
$$
Notice that $R^1f_*\O_X(A-D_0-D_1)$ is the $\mu_m$-eigensheaf of $R^1(f\circ \pi)_*F_0j'_*\mathbb D^H_{X'\smallsetminus R'}$, since we have the isomorphism $\omega_X(R)\tensor L\isom\O_X(A-D_0-D_1)$ and Remark \ref{rmk: eigensheaf decomposition}. Applying Proposition \ref{prop: torsion freeness of direct image of MHM} with Corollary \ref{cor: cyclic cover MHM on X'}, every associated prime of $R^1f_*\O_X(A-D_0-D_1)$ is the $f$-image of the generic point of a stratum of $(X,R)$. Hence, the associated primes are not contained in $Z_0$ by Remark \ref{rmk: Ambro}. Since $f_*\O_{D_0}(A-D_1)$ is supported on $Z_0$, the second morphism in this exact sequence is a zero map. This induces a surjection
$$
f_*\O_X(A-D_1)\twoheadrightarrow f_*\O_{D_0}(A-D_1).
$$

From the assumption, we have a natural isomorphism $\O_Y\isom f_*\O_X(A)$, and thus, a natural isomorphism $\I_{Y,Z_1}\isom f_*\O_X(A-D_1)$. Accordingly, $\I_{Y,Z_0}$ is the kernel of this surjection, and from $\I_{Z_0,Z_1}\isom \I_{Y,Z_1}/\I_{Y,Z_0}$ we obtain a natural isomorphism
$$
\I_{Z_0,Z_1}\isom f_*\O_{D_0}(A-D_1).
$$

Let $\left\{Z_{0,k}\right\}$ be the set of irreducible components of $Z_0$, and $D_{0,k}$ be the sum of divisors in both $D_0$ and $f^{-1}(Z_{0,k})$. Then, it is easy to see that $D_{0,k}\cap D_{0,k'}=\emptyset$ if $k\neq k'$ by Remark \ref{rmk: Ambro} applied to $D_0\cup D_1$, and thus $D_0=\coprod_k D_{0,k}$ which implies a natural isomorphism
\begin{equation}
\label{eqn: I isom O(A-D)}
\I_{Z_0,Z_1}\isom \bigoplus_k f_*\O_{D_{0,k}}(A-D_1).
\end{equation}

Consider the mixed Hodge module $\M'_{D_{0,k}}\in MHM(X')$ in Proposition \ref{prop: cyclic cover MHM} by letting $D=D_{0,k}$ and $R=R$. Then, $\O_{D_{0,k}}(A-D_1)$ is the $\mu_m$-eigensheaf of $\pi_*F_0\M'_{D_{0,k}}$, because $\omega_{D_{0,k}}(R)\tensor L\isom \O_{D_{0,k}}(A-D_1)$ (adjunction formula) and Remark \ref{rmk: eigensheaf decomposition}.

Using Proposition \ref{prop: cyclic cover MHM} on the top weight nonzero graded piece of $\M'_{D_{0,k}}$, we have the following surjection
\begin{equation}
\label{eqn: mhm surjection}
\M'_{D_{0,k}}\twoheadrightarrow \mathbb D^H_{(D'_{0,k}\cup R')_{I_k}}   
\end{equation}
in $MHM(X')$ for some index set ${I_k}$ which induces the minimal stratum $(D'_{0,k}\cup R')_{I_k}$ of $D'_{0,k}\cup R'$ contained in $D'_{0,k}$. Recall that $(D'_{0,k}\cup R')_{I_k}$ has cyclic quotient singularities \cite{EV92}*{Lemma 3.24}, so that we have \eqref{eqn: rationally smooth} and the following first nonzero term in the Hodge filtration
$$
F_0\mathbb D^H_{(D'_{0,k}\cup R')_{I_k}}\isom\omega_{(D'_{0,k}\cup R')_{I_k}},
$$
using Lemma \ref{lem: basic IC}. Hence, we have the surjection of the first nonzero terms in the Hodge filtrations of \eqref{eqn: mhm surjection}:
$$
\omega_{D_{0,k}'}(R')\twoheadrightarrow \omega_{(D'_{0,k}\cup R')_{I_k}}.
$$
Since the second object is a reflexive $\O_{(D'_{0,k}\cup R')_{I_k}}$-module, this induces an isomorphism
$$
\omega_{D_{0,k}'}(R')|_{(D'_{0,k}\cup R')_{I_k}}\isom \omega_{(D'_{0,k}\cup R')_{I_k}}.
$$
Indeed, this can be checked using the local description in \eqref{eqn: cyclic cover hypersurface}. In other words, there exists an isomorphism
$$
F_0\M'_{D_{0,k}}|_{(D'_{0,k}\cup R')_{I_k}}\isom F_0\mathbb D^H_{(D'_{0,k}\cup R')_{I_k}}.
$$
The pushforward by $\pi$ admits the $\mu_m$-eigensheaves decomposition, because the subvariety $(D'_{0,k}\cup R')_{I_k}\subset X'$ is invariant under the cyclic group $\mu_m$-action. Therefore, we obtain a surjection
\begin{equation}
\label{eqn: eigensheaves of the first HF}
\O_{D_{0,k}}(A-D_1)\twoheadrightarrow \O_{(D_{0,k}\cup R)_{I_k}}(A-D_1) 
\end{equation}
of $\mu_m$-eigensheaves of $\pi_*F_0\M'_{D_{0,k}}\twoheadrightarrow \pi_*F_0\mathbb D^H_{(D'_{0,k}\cup R')_{I_k}}$.

Apply Lemma \ref{lem: commutativity of direct image and first HF} to the direct image $(f\circ\pi)_*$ of $\M'_{D_{0,k}}\twoheadrightarrow \mathbb D^H_{(D'_{0,k}\cup R')_{I_k}}$:
\begin{equation*}
\xymatrix{
{(f\circ \pi)_*F_{0}\M'_{D_{0,k}}}\ar@{<->}[r]^-{\sim}\ar[d]& {F_{0}\H^0\left((f\circ \pi)_*\M'_{D_{0,k}}\right)}\ar[d]\\
{(f\circ \pi)_*F_{0}\mathbb D^H_{(D'_{0,k}\cup R')_{I_k}}}\ar@{<->}[r]^-{\sim}&{F_{0}\H^0\left((f\circ \pi)_*\mathbb D^H_{(D'_{0,k}\cup R')_{I_k}}\right)}
}
\end{equation*}
Note that $\mathbb D^H_{(D'_{0,k}\cup R')_{I_k}}$ is a polarizable Hodge module of weight $-n+|I_k|$. Hence, the morphism
$$
\H^0\left((f\circ \pi)_*\M'_{D_{0,k}}\right)\xrightarrow{\psi}\H^0\left((f\circ \pi)_*\mathbb D^H_{(D'_{0,k}\cup R')_{I_k}}\right)
$$
is a morphism from a mixed Hodge module to a pure Hodge module. As a consequence, the right term decomposes into $\mathrm{image}(\psi)\oplus \mathrm{coker}(\psi)$. In particular, considering $F_0$ of this decomposition, we obtain a morphism
$$
(f\circ \pi)_*F_{0}\M'_{D_{0,k}}\xrightarrow{\rho}(f\circ \pi)_*F_{0}\mathbb D^H_{(D'_{0,k}\cup R')_{I_k}}
$$
such that the right term decomposes into $\mathrm{image}(\rho)\oplus \mathrm{coker}(\rho)$. Since $\rho$ is $\mu_m$-equivariant, this decomposition can be modified to be $\mu_m$-equivariant by an elementary argument. Therefore,
the morphism of $\mu_m$-eigensheaves
$$
f_*\O_{D_{0,k}}(A-D_1)\xrightarrow{\rho_m} f_*\O_{(D_{0,k}\cup R)_{I_k}}(A-D_1)
$$
induced by \eqref{eqn: eigensheaves of the first HF} has the decomposition $\mathrm{image}(\rho_m)\oplus \mathrm{coker}(\rho_m)$ of the right term.

Consider a section $s$ of $f_*\O_{D_{0,k}}(A-D_1)$. This is also a section of $\I_{Z_0,Z_1}$ by \eqref{eqn: I isom O(A-D)}. If $s$ is in the kernel of $\rho_m$, then this is lifted to the section of $\O_{D_{0,k}}(A-D_1)$ vanishing along $(D_{0,k}\cup R)_{I_k}$. From Remark \ref{rmk: Ambro}, $(D_{0,k}\cup R)_{I_k}$ maps surjectively onto $Z_{0,k}$, and thus $s$ vanishes along $Z_{0,k}$ as a section of $\I_{Z_0,Z_1}$. This implies that $s$ is zero. Therefore, $\rho_m$ is injective, and we have a left inverse of $\rho_m$ obtained from the decomposition.

Recall by Saito's Decomposition Theorem and Lemma \ref{lem: commutativity of direct image and first HF}, we have a non-canonical isomorphism
$$
R(f\circ \pi)_*F_{0}\mathbb D^H_{(D'_{0,k}\cup R')_{I_k}}\isom \bigoplus_i R^i(f\circ \pi)_*F_{0}\mathbb D^H_{(D'_{0,k}\cup R')_{I_k}}[-i].
$$
In particular, taking the $\mu_m$-eigensheaf of $\pi_*F_{0}\mathbb D^H_{(D'_{0,k}\cup R')_{I_k}}$, we obtain a left quasi-inverse of the natural morphism
$$
f_* \O_{(D_{0,k}\cup R)_{I_k}}(A-D_1)\to  Rf_*\O_{(D_{0,k}\cup R)_{I_k}}(A-D_1).
$$
Combined with the left inverse of $\rho_m$, we obtain a left quasi-inverse of the natural morphism
$$
f_*\O_{D_{0,k}}(A-D_1)\to Rf_*\O_{(D_{0,k}\cup R)_{I_k}}(A-D_1).
$$
The direct sum of this for all the irreducible components $Z_{0,k}$ deduces a left quasi-inverse of the natural morphism
$$
\I_{Z_0,Z_1}\to \bigoplus_k Rf_*\O_{(D_{0,k}\cup R)_{I_k}}(A-D_1).
$$
See \eqref{eqn: I isom O(A-D)} for the first term. This factors through $\bigoplus_k Rf_*\O_{(D_{0,k}\cup R)_{I_k}}(-D_1)$. Furthermore, $D_1$ is an snc divisor when restricted to the smooth variety $(D_{0,k}\cup R)_{I_k}$, and $D_1$ maps to $Z_1$. Let $W$ be the disjoint union of $(D_{0,k}\cup R)_{I_k}$ for all $k$.

In conclusion, $\I_{Z_0,Z_1}\to Rf_*\O_W(-D_1|_W)$ admits a left quasi-inverse, so that the converse statement of Corollary \ref{cor: criterion for log rational pair} applies to the morphism $f|_W:(W,D_1|_W)\to (Z_0,Z_1)$. Therefore, $(Z_0,Z_1)$ is a log rational pair.
\end{proof}

\begin{rmk}
The analysis of the associated primes of $R^1f_*\O_X(A-D_0-D_1)$ to deduce \eqref{eqn: I isom O(A-D)} draws inspiration from the arguments of Ambro \cite{Ambro11}. The key addition is the use of the mixed Hodge module $j'_*\mathbb D^H_{X'\smallsetminus R'}$ in Corollary \ref{cor: cyclic cover MHM on X'}.
\end{rmk}

We complete the proof of Theorem \ref{thm: lc implies DB}.

\begin{proof}[Proof of Theorem \ref{thm: lc implies DB}]
We reduce to the setup of Theorem \ref{thm: lc admits log rational stratification} via Chow's lemma and a log resolution of singularities $\mu:(\widetilde X,\widetilde \Delta) \to (X,\Delta)$ such that $\mu^*(K_X+\Delta)=K_{\widetilde X}+\widetilde \Delta$. Indeed, there is an inclusion $\mu_*\O_{\widetilde X}(\lceil-\widetilde\Delta^{<1}\rceil)\to \O_X(\lceil-\Delta^{<1}\rceil)$, and thus the natural morphism 
$$
\O_Y\to(f\circ\mu)_*\O_{\widetilde X}(\lceil-\widetilde\Delta^{<1}\rceil)
$$
is an isomorphism. In addition, we have $K_{\widetilde X}+\widetilde \Delta\sim_{\Q,f\circ \mu}0$.

Let $W\subset Y$ be a union of $f$-images of log canonical centers of $(X,\Delta)$. Then we have a proper subvariety $W_1\subset W$ which is the union of $f$-images of log canonical centers, not mapping surjectively onto an irreducible component of $W$. Notice that $W_1=\emptyset$ if every $f$-image of a log canonical center is not properly contained in $W$ (i.e. $W$ is minimal).

The pair $(W,W_1)$ is log rational by Theorem \ref{thm: lc admits log rational stratification}, and inductively, we obtain a sequence of subvarieties
$$
\emptyset=W_{k+1}\subsetneq W_{k}\subsetneq \dots\subsetneq W_1\subsetneq W_0=W
$$
such that $(W_{i},W_{i+1})$ is a log rational pair for all $i\in [0,k]$. Therefore, if $W$ is minimal, then $W$ has rational singularities. In general, we inductively conclude that $W$ has Du Bois singularities, using the fact that $W_i$ has Du Bois singularities if $(W_i,W_{i+1})$ is a Du Bois pair and $W_{i+1}$ has Du Bois singularities (see \cite{Kovacs11}*{Proposition 5.1}).
\end{proof}

We emphasize the following variant of the fact that log canonical singularities are Du Bois, which is sometimes more useful in practice; see e.g. the proof of Corollary \ref{cor: log canonical model is Du Bois} below.

\begin{cor}
\label{cor: pushforward pluricanonical is Du Bois}
Let $(X,D)$ be a log canonical pair (resp. klt pair) with a boundary $\Q$-divisor $D$ and $f:X\to Y$ be a proper morphism between normal varieties, with connected fibers. Suppose $f_*\O_X(N(K_X+D))$ is a line bundle for some positive integer $N$ such that $ND$ is an integral divisor. Then $Y$ has Du Bois singularities (resp. rational singularities).

Furthermore, there exists an increasing sequence of subvarieties
$$
\emptyset=Y_{k+1}\subsetneq Y_{k}\subsetneq \dots\subsetneq Y_1\subsetneq Y_0=Y
$$
such that each $Y_i$ is a union of $f$-images of log canonical centers of $(X,D)$, and $(Y_i,Y_{i+1})$ is a log rational pair for all $i\in [0,k]$.
\end{cor}

Here, $\O_X(N(K_X+D))$ is a reflexive sheaf associated to the integral Weil divisor $N(K_X+D)$.

\begin{proof}
Let a line bundle $L\isom f_*\O_X(N(K_X+D))$. By adjointness, we have a morphism
$$
f^*L\to \O_X(N(K_X+D))
$$
which induces an effective integral divisor $A$ such that
$$
f^*L\isom \O_X(N(K_X+D)-A).
$$
By the projection formula, we have a natural isomorphism $L\isom L\tensor f_*\O_X(A)$. Therefore, the natural morphism $\O_Y\to f_*\O_X(A)$ is an isomorphism.

Denote $\Delta:=D-\frac{1}{N}A$. Since $N(K_X+D)-A$ is a Cartier divisor, $K_X+\Delta$ is a $\Q$-Cartier divisor. Consequently, $(X,\Delta)$ is a log canonical pair with a sub-boundary $\Q$-divisor $\Delta$. Additionally, $K_X+\Delta\sim_{\Q,f}0$ and $f_*\O_X(\lceil\frac{1}{N}A\rceil)=\O_Y$ are immediate. Then the proof of Theorem \ref{thm: lc implies DB} works verbatim.

From the fact that the discrepancy of a divisor of $(X,\Delta)$ is at least the discrepancy of that of $(X,D)$, the log canonical centers of $(X,\Delta)$ are the log canonical centers of $(X,D)$. Therefore, each $Y_i$ is a union of $f$-images of log canonical centers of $(X,\Delta)$, hence that of $(X,D)$.
\end{proof}

We end this section with a proof that the log canonical model of a log canonical pair is Du Bois.

\begin{proof}[Proof of Corollary \ref{cor: log canonical model is Du Bois}]
Since the log canonical ring is finitely generated, there exists a sufficiently large and divisible $N$ such that
$$
R^{(N)}(X,D)=\bigoplus_{m\ge 0} H^0(X,Nm(K_X+D))
$$
is generated by $H^0(X,N(K_X+D))$. Then the complete linear series $|N(K_X+D)|$ induces the rational map
$$
X\dashrightarrow \mathrm{Proj}\; R(X,D)
$$
with a very ample divisor $H$ on $\mathrm{Proj}\; R(X,D)$. We resolve the indeterminacy of this rational map via a resolution of singularities $\mu:\widetilde X\to X$. Let $K_{\widetilde X}+\widetilde \Delta=\mu^*(K_X+D)$ and take $\widetilde D:=\widetilde\Delta^{\ge 0}$ (i.e. the effective part of $\widetilde \Delta$). Therefore we have a natural isomorphism
$$
\mu_*\O_{\widetilde X}(Nm(K_{\widetilde X}+\widetilde D))=\O_X(Nm(K_X+D))
$$
for all $m\ge 1$, so that $R^{(N)}(\widetilde X,\widetilde D)=R^{(N)}(X,D)$.

From the construction, the morphism
$$
f:\widetilde X\to \mathrm{Proj}\; R(X,D)
$$
is induced by the complete linear series $|N(K_{\widetilde X}+\widetilde D)|$, and thus, there exists a natural morphism
$$
f^*H\to \O_{\widetilde X}(N(K_{\widetilde X}+\widetilde D)).
$$
By adjointness, we have a short exact sequence
$$
0\to H\to f_*\O_{\widetilde X}(N(K_{\widetilde X}+\widetilde D))\to Q\to 0.
$$
Twist by $H^{\tensor N'}$ for large enough $N'$:
\begin{equation}
\label{eqn: SES for log canonical model}
0\to H^{\tensor N'+1}\to f_*\left(f^*H^{\tensor N'}\tensor\O_{\widetilde X}(N(K_{\widetilde X}+\widetilde D))\right)\to Q\tensor H^{\tensor N'}\to 0.
\end{equation}
Notice that $H^0(\mathrm{Proj}\; R(X,D), H^{\tensor N'+1})=H^0(\widetilde X, \O_{\widetilde X}(N(N'+1)(K_{\widetilde X}+\widetilde D)))$ which factors through $H^0(\widetilde X, f^*H^{\tensor N'}\tensor\O_{\widetilde X}(N(K_{\widetilde X}+\widetilde D)))$. Therefore from the long exact sequence of sheaf cohomologies of \eqref{eqn: SES for log canonical model}, we conclude
$$
H^0(\mathrm{Proj}\; R(X,D), Q\tensor H^{\tensor N'})=0
$$
for large enough $N'$. Therefore, $Q=0$. This implies that 
$$
H\isom f_*\O_{\widetilde X}(N(K_{\widetilde X}+\widetilde D)).
$$
Thus, Corollary \ref{cor: pushforward pluricanonical is Du Bois} completes the proof; when $(\widetilde X,\widetilde D)$ is a log canonical pair (resp. klt pair), then $\mathrm{Proj}\; R(X,D)$ has Du Bois singularities (resp. rational singularities).
\end{proof}

\end{document}